\documentclass[10pt,reqno]{article}
\usepackage{latexsym}
\usepackage{amssymb}
\usepackage{amsmath}
\usepackage[all]{xy}

\newtheorem{thm}{Theorem}[section]
\newtheorem{defin}{Definition}[section]
\newtheorem{prop}{Proposition}[section]
\newtheorem{lemma}{Lemma}[section]
\newtheorem{cor}{Corollary}[section]

\newcommand{\Mbar}{\overline {M}}
\newcommand{\Sbar}{\overline {S}}
\newcommand{\Ric}{\text{Ric}}
\newcommand{\Ker}{\text{Ker}}
\newcommand{\jbar}{\bar {j}}
\newcommand{\vbar}{\bar {v}}
\newcommand{\zbar}{\bar {z}}
\newcommand{\lbar}{\bar {l}}
\newcommand{\nbar}{\bar {n}}
\newcommand{\kahler}{K\"ahler     }
\newcommand{\KE}{K\"ahler-Einstein     }
\newcommand{\qed}{\hfill \fbox{}}

\def\NN {{\mathbb N}}

\def\RR {{\mathbb R}}
\def\CC {{\mathbb C}}

\def\Om{\Omega}

\def\la{\lambda}
\def\de{\delta}
\def\om{\omega}
\def\ga{\gamma}
\def\ve{\varepsilon}

    \def\cU{{\cal U}}

\newenvironment{proof}{\medskip
\noindent{\bf Proof: }}{{\hfill$\square$}{\medskip}}

\title{On the Asymptotic Expansion of Complete Ricci-flat \kahler metrics on
quasi-projective manifolds}

\author{Bianca Santoro}
\date{}

\begin{document}
\maketitle

\begin{abstract}
In this work, we describe the asymptotic behavior of complete metrics
with prescribed Ricci curvature on open \kahler manifolds that can 
be compactified by the addition of a smooth and ample divisor. First, we
construct a explicit sequence of \kahler metrics with special
approximating properties. 
Using those metrics as starting point, we are able to work out 
the asymptotic behavior of the solutions
given in \cite{TY1}, in particular obtaining their full asymptotic
expansion.
\end{abstract}

\section{Introduction}
\label{s.introduction}

In 1978, Yau \cite{Y1} proved the Calabi Conjecture by showing the
existence and uniqueness of \kahler metrics with prescribed Ricci curvature 
on compact complex manifolds.

Following this work, Tian and Yau \cite{TY1} settled the non-compact version
of Calabi's Conjecture on quasi-projective manifolds that can be compactified
by adding a smooth, ample divisor. In a subsequent work (\cite{TY2}),
their result was extended to the case where the divisor has 
multiplicity greater than one and orbifold-type singularities.
This generalization was independently done by Bando \cite{Bando} and
Kobayashi \cite{K}.

Once the existence problem is solved, an interesting question that arises 
concerns the behavior of these complete metrics near the divisor. 
This question is also addressed to by Tian and Yau (\cite{TY1}).

In a subsequent work, Tian and Yau (\cite{TY2}) obtain existence results, as
well as the asymptotics of their solution, for the
case in which the divisor has multiplicity strictly greater than one. 
These techniques, however, make an
essential use of the higher multiplicity of the divisor and thus they
do not work when this multiplicity equals one which can be thought of
as the generic case. In particular the asymptotics properties of the
solutions found in (\cite{TY1}) remained ignored.

The aim of this paper is to provide
an answer to this question, therefore refining the 
main result in \cite{TY1}. 
More precisely, we shall first construct a sequence
of complete \kahler metrics with special approximating properties 
on a quasi-projective manifold (in our case, the complement
of  a smooth, ample divisor on a compact complex manifold). Then by
using these approximating metrics, we are going to
study the solution of a complex 
Monge-Amp\`ere equation on the open manifold. A careful analysis 
of the complex Monge-Amp\`ere operator will allow us to describe
the asymptotic properties of the solution. As a matter of fact,
the reader will note that
our results apply equally well to divisors having orbifold-type singularities.

In a subsequent work, we expect to remove the smoothness assumption 
on the divisor so as to include divisors having normal crossings.

To state the main results of this paper, let us consider a 
compact, complex manifold $\Mbar$ of complex dimension $n$. Let
$D$ be an {\em admissible} divisor in  $\Mbar$, ie, a divisor satisfying the
following conditions:
\begin{itemize}
\item
Sing{ $\Mbar$}$\subset D$.
\item
$D$ is smooth in $\Mbar\setminus $ Sing{ $\Mbar$}.
\item
For every $x \in$ Sing{ $\Mbar$}, the corresponding
local uniformization $\Pi_x : \tilde{\cal U}_x \rightarrow
{\cal U}_x$, with $\tilde{\cal U}_x \in \CC^n$, is such that
$\Pi_x^{-1}(D)$ is smooth in $\tilde{\cal U}_x $.
\end{itemize}

Let $\Om$ be a  smooth, closed $(1,1)$-form in the 
cohomology class $c_1(K_{\Mbar}^{-1} \otimes L_D^{-1})$, where
$K_{\Mbar}$ stands for the Canonical Line Bundle of $\Mbar$, and
$L_D$ for the line bundle associated to $D$.
Let $S$ be a defining section of $D$ on $L_D$ 
and let $M$ be the open manifold $M = \Mbar \setminus D$.
Consider a hermitian metric $||.||$ on $L_D$.

Fefferman, in his paper \cite{Fef}, developed inductively an
\hbox {$n$-th} order approximation to a complete \KE metric on strictly pseudoconvex domains on 
$\mathbb{C}^n$ with smooth boundary, and he suggested that higher order approximations could be 
obtained by considering $\log$ terms in the formal expansion of the solution to a certain complex Monge-Amp\`ere 
equation. This idea was used by Lee and
Melrose in \cite{Mel}, where they constructed the full asymptotic expansion of the solution to the
Monge-Amp\`ere equation introduced by Fefferman.

Motivated by this work, we construct inductively a sequence of 
rescalings $||.||_{\phi_m} := e^{\phi_m/2} ||.||$ of a fixed hermitian metric $||.||$ on $L_D$,
which will be the main ingredient of the proof of the following result.

\begin{thm}
\label{t.approximatingmetrics}
Let $M$, $\Om$ and $D$ be as above. Then, for every $\ve >0$, there exists an
explicitly given complete
\kahler metric $g_\ve$ such that
\begin{equation}
\label{e.MAve}
\Ric(g_\ve) - \Om = \frac{\sqrt{-1}}{2\pi} \partial \bar\partial f_\ve 
\hspace{1.5cm} \text{on $M$,}
\end{equation}
where $f_\ve$ is a smooth function on $M$ such that all 
its covariant derivatives decay to the order of $O(||S||^\ve)$.
Furthermore, for any $k \geq 0$, 
the norm of the $k$-th covariant derivative of the Riemann curvature 
tensor $R(g_\ve)$ of the metric $g_\ve$ decays 
to the order of $O((-n \log ||S||^2)^{-\frac{k+2}{2n}})$. 

\end{thm}

\noindent {\bf Remark}: In the above statement, it should be emphasized that
the metric in question is constructed inductively in an explicit
manner. In other words, this
result provides complete metrics that are ``approximate solutions''
to the Calabi problem, but that have the advantage of being
explicitly described.

\vspace{0.1cm}

So far,
there has been a large amount of work concerned with 
deriving asymptotic expansions for 
\KE metrics in different contexts: after Cheng and Yau \cite{CY1} 
proved existence and uniqueness of \KE metrics
on strictly pseudoconvex domains in $\mathbb{C}^n$ with smooth boundary
(in addition to results on 
the regularity of the solution), Lee and Melrose \cite{Mel} derived an 
asymptotic expansion for the Cheng-Yau solution, which completely determines
the form of the singularity and improves the regularity result of \cite{CY1}.
On the setting of quasi-projective manifolds, Cheng and Yau \cite{CY2} and 
Tian and Yau \cite{TY3} showed the existence of \KE metrics under certain conditions 
on the divisor, and Wu  developed the asymptotic expansion to the
Cheng-Yau metric on a quasi-projective manifold (also assuming some conditions
on the divisor), as the parallel part to the work 
of Lee and Melrose \cite{Mel}.

However the asymptotic description of  
complete \kahler  metrics with prescribed Ricci curvature, which are not of
\KE type, was still lacking, for example in the context of quasi-projective
manifolds considered in (\cite{TY1}). This description will be provided by our
results below.

In \cite{TY1}, the result of existence of a complete \kahler metric (in a given \kahler class) 
with prescribed Ricci curvature 
is achieved by solving the following complex Monge-Amp\`ere equation
\begin{equation}
\label{MA1}
\begin{cases}
\left(
\om + \frac{\sqrt{-1}}{2\pi}\partial\bar\partial u
\right)^n  = e^f \om^n, & \\
\om + \frac{\sqrt{-1}}{2\pi}\partial\bar\partial u >0,
& \text{$u \in C^\infty(M, \RR)$},
\end{cases}
\end{equation}
where $f$ is a given smooth function
satisfying the integrability 
condition
\begin{equation}
\label{intcondition1}
\int_M (e^f-1)w^n = 0.
\end{equation}

Our main result describes the asymptotic behavior of the solution to
(\ref{MA1}), by showing that the 
approximate metrics given in Theorem \ref{t.approximatingmetrics} are
asymptotically as close to the 
actual solution as possible.

We should point out here that the underlying analysis of the case 
we are considering, as well as the methods 
used in our work, differs fundamentally from Wu's work (\cite{DaminWu}). 
Namely, 
it is not trivial to show, for example, that the solution $u(x)$ 
to (\ref{MA1}) vanishes uniformly when $x$ approaches infinity. In the 
\KE context considered by Wu, it is not hard to see it directly from the 
defining equation (as observed by Cheng and Yau in \cite{CY1}).
Also, the $\log$-filtration of the Cheng-Yau H\"older ring considered by Wu 
is not preserved in our case.

\begin{thm}
\label{t.mainthm}
For each $\ve>0$, let $g_\ve$ and $f_\ve$ be given by Theorem \ref{t.approximatingmetrics}.

Consider the  problem
\begin{equation}
\label{MA1ve}
\begin{cases}
\left(
\om_{g_\ve} + \frac{\sqrt{-1}}{2\pi}\partial\bar\partial u_\ve
\right)^n  = e^{f_\ve} \om_{g_\ve}^n, & \\
\om_{g_\ve} + \frac{\sqrt{-1}}{2\pi}\partial\bar\partial u_\ve >0,
& \text{$u_\ve \in C^\infty(M, \RR)$}.
\end{cases}
\end{equation}

Then, there exists a  solution $u_\ve (x)$ which decays 
to the order of at least  
$O(||S||^\ve(x))$ for $x$ sufficiently close to $D$.

Moreover, the norm of the $k$-th covariant derivative $|\nabla^k u_\ve|_{g_\ve}$ 
of $u_\ve$ decays as $O(||S||^\ve(x) \rho_{g_\ve}^{-\frac{k+2}{n+1}})$,
where $\rho_{g_\ve}(x)$ denotes the distance, with respect to the metric
$g_\ve$, from a fixed point $x_0 \in M$ to 
$x$.

\end{thm}

This theorem has an important, straightforward corollary.
\begin{cor}
\label{c.RicciFlat}
Let $\Mbar$ be a compact \kahler manifold of complex dimension $n$, and let $D$ be
a smooth anti-canonical divisor. Then for any $\ve>0$, there
exists a complete Ricci-flat \kahler metric 
on $M = \Mbar \setminus D$ that can be described as 
$$
\hat{\om} = \om_{g_\ve} + \om_{u_\ve},
$$
where $g_\ve$ is the \kahler metric constructed in Theorem \ref{t.approximatingmetrics}, 
and $\om_{u_\ve}$ is a \kahler form that 
decays at least to the order of $ O(||S||^{\ve})$ when $x$ approaches the divisor.
Therefore, $\om_{g_\ve}$ provides the asymptotics of the Ricci-flat metric $\hat{\om}$. 
\end{cor}

The structure of the paper is as follows.
In Section \ref{s.approx}, we construct inductively a sequence of hermitian metrics 
$\{ ||.||_m \}_{m \in \NN}$ on $L_D$ such that the closed $(1,1)$-form 
\begin{equation}
\label{omegam}
\om_m = \frac{\sqrt{-1}}{2\pi} \frac{n^{1 + 1/n}}{n+1} \partial \bar{\partial}
(-\log||S||_m^2)^{ \frac{n+1}{n}}
\end{equation}
is positive definite on a tubular neighborhood $V_m$ of $D$ in $\Mbar$.

The \kahler form $\om_m$ defines a \kahler metric $g_m$ on $V_m$ such that 
$\Ric(g_m) - \Om = \frac{\sqrt{-1}}{2\pi} \partial \bar \partial f_m$, 
for a smooth function $f_m$ on $M$ that decays to the order of $||S||^m$.
An important technical result for this construction is Lemma~\ref{chato},
whose proof is the object of Section~\ref{s.proofchato}.

In Section~\ref{s.whole}, we use the constructions of Section \ref{s.approx}
to complete the proof of Theorem \ref{t.approximatingmetrics}. First, we shall
obtain the necessary estimates on the decay of the
Riemann curvature tensor of the metrics $g_m$. Then we shall proceed to the
construction of approximating metrics that are defined on the whole manifold
(and not only on a neighborhood of the divisor at infinity).

Finally, Section \ref{s.MA} is devoted to the asymptotic study of the Monge-Amp\`ere 
equation (\ref{MA1}). By using the maximum principle for the complex 
Monge-Amp\`ere operator, along with the construction of a suitable barrier, we
shall complete the proof of Theorem \ref{t.mainthm}.

\noindent
{\bf Acknowledgements:}
I would like to thank my advisor, Prof. Gang Tian, for all the encouragement during 
my Ph.D. years at MIT, and for suggesting to me 
this topic of research, 
and to Prof. Julio Rebelo, for the careful reading of many preliminary versions of this
work.

\section{Approximating \kahler metrics}
\label{s.approx}
Let $\Mbar$ be a compact \kahler manifold of complex dimension $n$, and let $D$ be
an admissible divisor in  $\Mbar$.

The divisor $D$  induces  a line bundle $L_D$ on  $\Mbar$.
We will assume that the restriction of $L_D$ to $D$ is ample, so that 
there exists an orbifold hermitian metric $||.||$ on $L_D$ such that its curvature form 
$\tilde\om$ is positive definite along $D$.  
  
Consider a closed $(1,1)$-form $\Om$ in the Chern class $c_1(-K_{\Mbar} - L_D)$. 
The  
goal of this section  is to construct a complete k\"ahler metric $g$ such that

\begin{equation}
\label{problem}
\text{Ric}(g) - \Om = \frac{\sqrt{-1}}{2\pi}\partial \bar{\partial}f  \\ \text{      on $M$},
\end{equation}
 for a smooth function $f$ with sufficiently fast decay, where $M = \Mbar \setminus D$
and  $\Ric(g)$ stands for the Ricci
form of the metric $g$.

Fix an orbifold hermitian metric  $||.||$
on $L_D$ such that its curvature form $\tilde\om$ 
is positive definite along $D$. 
We shall need to rescale the metric by a suitable factor which will 
be determined in the following discussion. Let us begin by observing that the
restriction $\Om|_D$ of $\Om$ to $D$ belongs to $c_1 (D)$ since, by
assumption, $\Om \in c_1(-K_{\Mbar} - L_D)$.
Hence, there exists a function $\varphi$ such that 
$\tilde{\om}|_D +  \frac{\sqrt{-1}}{2\pi} \partial \bar{\partial} \varphi $ 
defines a metric $g_D$ verifying $\Ric(g_D) = \Om|_D$.
So, by rescaling $||.||$ by an appropriate factor, we may assume 
that $\tilde\om$, when restricted to the 
infinity $D$, defines a metric $g_D$ such that $\Ric(g_D) = \Om|_D$.

Next denote by $S$ the defining section of $D$, and write 
$||.||_\phi  = e^{-\phi/2} ||.||$ for the rescaling 
of $||.||$, where $\phi$ is any smooth function on $\Mbar$.

We define

\begin{equation}
\label{omegaphi}
\om_\phi = \frac{\sqrt{-1}}{2\pi} \frac{n^{1 + 1/n}}{n+1} \partial \bar{\partial}
(-\log||S||_\phi^2)^{ \frac{n+1}{n}}.
\end{equation}

Then it follows that 
\begin{equation}
\label{relate}
\om_\phi = (-n\log||S||_\phi^2 )^{1/n}\tilde{\om}_\phi + 
\frac{1}{(-n\log||S||_\phi^2 )^{(n-1)/n}} \frac{\sqrt{-1}}{2\pi} \partial \log||S||_\phi^2
\wedge \bar\partial \log||S||_\phi^2,
\end{equation}
where $\tilde{\om}_\phi$ is the curvature form of the metric $||.||_\phi$.
From this expression, we can see that, as long as $\tilde{\om}_\phi$ is positive 
definite along $D$, ${\om_\phi}$ is positive definite near $D$.


We state here the main result of this section.

\begin{prop}
\label{t.inductive}
Let $\Mbar$ be a compact \kahler manifold of complex dimension $n$, and let $D$ be
an admissible divisor in  $\Mbar$. Consider also
a form  $\Om \in c_1(-K_{\Mbar} - L_D)$, where $L_D$ is
the line bundle induced by $D$.

Then there exist sequences of neighborhoods $\{V_m\}_{m \in \NN}$ of $D$
along with complete K\"ahler 
metrics $\om_m$  on $(V_m\setminus D, \partial(V_m\setminus D))$ (as defined in (\ref{omegaphi}))
such that
\begin{equation}
\label{e.thmseqs}
\text{Ric}(\om_m) - \Om = \frac{\sqrt{-1}}{2\pi}\partial \bar{\partial}f_m  \\ \text{  
on $V_m\setminus D$}
\end{equation}
where $f_m$ are smooth functions on $M = \Mbar \setminus D$. Furthermore 
each $f_m$ decays to the order of $O(||S||^m)$. In addition
the curvature tensors $R(g_m)$ of the metrics  $g_m$
decay at least to the order of $(-n\log||S||^2)^{\frac{-1}{n}}$ near the divisor.
\end{prop}

\medskip
The remainder of this section will be devoted to the proof of Proposition~\ref{t.inductive}.

If $\tilde\om = - \frac{\sqrt{-1}}{2\pi}\partial \bar{\partial}\log||S||^2$ 
is the curvature form of $||.||$, then 
for any \kahler metric $g'$ on $\Mbar$, 
$\text{Ric}(g') - \tilde\om  \in c_1(-K_{\Mbar} -L_D)$.
Hence, up to constant,
there is a unique function $\Psi$ such that
\begin{equation}
\label{acima}
\Om = \text{Ric}(g') - \tilde\om +  \frac{\sqrt{-1}}{2\pi} \partial \bar{\partial} \Psi. 
\end{equation}

\begin{defin}
\label{fphi}
For $x$ in the set where $\om_\phi$ is positive definite ($x$ near $D$), write
$$
f_\phi(x)= -\log||S||^2  - \log(\frac{\om_\phi^n}{{\om'}^n}) - \Psi,
$$
where $\om'$ is the K\"ahler form of $g'$.
\end{defin}

\begin{lemma}
The function $f_0(x)$ converges uniformly to a constant if and only if 
$\Ric(g_D) = \Om|_D$.
\end{lemma}

\begin{proof}
Choose a coordinate system $(z_1,\dots,z_n )$ around a point $x$ near $D$ such that 
the local defining section $S$ of $D$ is given by $\{ z_n = 0\}$.
In these coordinates, write $\tilde{\om} = \tilde\om_0$ as $(h_{ij})_{1\leq i,j \leq n}$, 
$g'$ as  $(g'_{ij})_{1\leq i,j \leq n}$, and $||.||$ as a positive function $a$.

By definition we have
\begin{equation}
f_0(x)= - \log\left(\frac{||S||^2\om_0^n}{{\om'}^n}\right) - \Psi(x)
=
- \log\left(\frac{a \det(h_{ij})_{1\leq i,j \leq n-1} }{\det (g'_{ij})_{1\leq i,j \leq 
n}}\right)(x) 
- \Psi(x)
+ O(||S(x)||),
\end{equation}
for $x$ near $D$.

Since  $a^{-1}\det (g'_{ij})_{1\leq i,j \leq n}|_D$ is a well-defined volume form on $D$,
it makes sense to set
$$
f_0(x)= - \log(\frac{a \det(h_{ij})_{1\leq i,j \leq n-1}
e^\Psi}{\det (g'_{ij})_{1\leq i,j \leq n}})(z',0)
+ O(||S(x)||),
$$
for $x = (z', z_n)$.
Hence, $\lim_{x \rightarrow D}f_0(x)$ is a constant if and only if 
$\frac{a \det(h_{ij})_{1\leq i,j \leq n-1} e^\Psi}{\det (g'_{ij})_{1\leq i,j \leq n}}(z',0)$
is constant.
In other words, $\lim_{x \rightarrow D}f_0(x)$ is a constant if and only if 
$$
\frac{\sqrt{-1}}{2\pi} \partial \bar{\partial} \Psi = -\log\det(h_{ij})_{1\leq i,j \leq n-1} 
- \log a \det(g'_{ij})_{1\leq i,j \leq n}.
$$
Since $\Om = \text{Ric}(g') - \tilde\om +  \frac{\sqrt{-1}}{2\pi} \partial \bar{\partial} \Psi, $
$\lim_{x \rightarrow D}f_0(x)$ is a constant if and only if $\Ric(g_D) = \Om|_D$.
\end{proof}

An appropriate choice of $\Psi$ allows us to assume in the sequel that
$f_0(x)$ converges uniformly to zero as $x \rightarrow D$.

The function $f_0(x)$ was only defined for $x$ near $D$, but we can extend it smoothly to 
be zero along $D$ since $||S||^2 \om_0^n $ is a well-defined volume form over all $\Mbar$.
Hence, there exists a $\delta_0 > 0$ such that, in the neighborhood 
$V_0 := \{x \in \Mbar ; ||S(x)|| <  \delta_0\}$, 
$f_0$ can be written as 
$$
f_0 = S \cdot u_1 + \Sbar \cdot \overline{u}_1,
$$
where $u_1$ is a $C^\infty$ local section in $\Gamma(V_0, L_D^{-1})$.

Our goal now would be to construct a function $\phi_1$ of the form 
$S \cdot \theta_1 + \Sbar \cdot \overline{\theta}_1$, so that the 
corresponding $f_{\phi_1} = f_1$ vanishes at order $2$ along $D$, and then
proceed inductively to higher order. Unfortunately, there is an obstruction
to higher order approximation that lies in the kernel of the laplacian on
$L_D^{-1}$ restricted to $D$. 
In order to deal with this difficulty, one must introduce $(-\log||S||^2)$-terms in 
the expansion of $\phi_1$, as pointed out in \cite{Fef} and \cite{Mel} where 
the similar
problem of finding expansions for the solutions of the Monge-Amp\`ere equation
on a strictly pseudoconvex domain was treated. Further details can be found below.

Following the techniques in \cite{TY2}, we are going to construct inductively
 a sequence of hermitian metrics $\{||.||_m\}_{m>0}$ on 
$L_D$ such that, for any $m>0$, there exists a $\de_m >0$ satisfying:
\begin{enumerate}
\item
The corresponding K\"ahler form $\om_m$ associated to $||.||_m$ (as defined in
(\ref{omegaphi})) is positive definite in $V_m := \{x \in \Mbar ; ||S(x)|| <  \delta_m\}$.
\item
The function $f_m$ associated to $\om_m$ (as in the Definition~\ref{fphi}) can be expanded 
in $V_m$ as
\begin{equation}
\label{expansionfm}
f_m = 
\sum_{k\geq m+1} \sum_{\ell =0}^{\ell_k} u_{k\ell}(-\log||S||^2_m)^\ell,
\end{equation}
where $u_{k\ell}$ are smooth functions on $\overline{V}_m$ that vanish to order 
$k$ on $D$. In particular the function $u_{k\ell}$ can
be written as
$$
u_{k\ell} = \sum_{i+j = k}S^i \Sbar^j \theta_{ij} + S^j \Sbar^i \overline{\theta}_{ij}, 
$$
for  $\theta_{ij} \in \Gamma(V_m, L_D^{-i} \bigotimes  \overline{L}_D^{-j})$.
\end{enumerate}

We define $||.||_0 = ||.||$, and it is clear that $||.||_0$ satisfies the Conditions
1 and 2 above. 
Now we proceed on the inductive step: assuming the existence
of $||.||_m$, we construct 
$||.||_{m+1}$. The next lemma gives a relation between $f_m$ and $f_\phi$, where
 $||.||_\phi = e^{-\phi/2}||.||_m $, and  
$f_\phi$ is 
associated to a smooth function $\phi$  on $V_m$ of the form
$$
\phi = (\sum_{i+j = k}S^i \Sbar^j \theta_{ij} + 
S^j \Sbar^i \overline{\theta}_{ij})(-\log||S||^2_m)^k, \hspace{1cm} 
\text{for $k \geq m+1$.} 
$$

\begin{lemma}
\label{chato}
Let $f_\phi(x)$ be defined as in Definition~\ref{fphi}, associated to 
$$ 
\om_\phi = \frac{\sqrt{-1}}{2\pi} \frac{n}{n+1} n^{1/n} \partial \bar{\partial}
(-\log||S||_m^2.\phi)^{ \frac{n+1}{n}}.
$$
Similarly let $f_m$ be associated to $\om_m$.
Then 
\begin{multline}
f_\phi = f_m + nm \phi
+ \frac{k\phi}{(-\log||S||_m^2)}\left( \frac{k-1}{(-\log||S||_m^2)} + (m-1) \right) + 
\\
+(-\log||S||_m^2)^{k-1}  \sum_{i+j = m+1}  
\left\{
ij \left( S^i  \Sbar^j \theta_{ij} + S^j \Sbar^i \overline{\theta}_{ij} \right)
+ \right. \\ \left. 
+ (-\log||S||_m^2)
\left[-2(n+1)j \left( S^i  \Sbar^j \theta_{ij} + S^j \Sbar^i \overline{\theta}_{ij}\right)
+ 
S^i  \Sbar^j \square_m \theta_{ij} 
+ S^j \Sbar^i \overline{\square}_m\overline{\theta}_{ij}
\right]
\right\} \\
+\sum_{k'\geq m+2} \sum_{\ell =0}^{\ell_{k'}} u_{k'\ell}(-\log||S||^2_m)^\ell
\end{multline}
where $\square_m = tr_{\om_D}(\overline{D_m}D_m)$ is the Laplacian
of the bundle $L_D^{-1}\bigotimes L_D^{-j}$ on $D$
with respect to the hermitian metric $||.||_m$, and 
the functions $u_{k'\ell}$ decay as $O(||S||^{k'})$.
\end{lemma}

The proof of this lemma will be postponed to the next section so that we
can now proceed to our inductive construction.

\noindent {\bf Proof of Proposition~\ref{t.inductive}}:
We want to find a function $\phi$ such that 
$||.||^2_{m+1} = e^{-\phi/2}||.||^2_{m}$ satisfies the Conditions 1 and 2, 
{\em ie}, we need to eliminate the terms
$\sum_{\ell =0}^{\ell_{m+1}} u_{m+1,\ell}(-\log||S||^2_m)^\ell$
from the expansion of $f_m$.   
Each of the $ u_{m+1,\ell}$, $0\leq \ell \leq m+1$
will  be eliminated successively, as follows.

\begin{description}
\item{Step 1:}
Write $u_{m+1,\ell_{m+1}}$ as 
$$ 
u_{m+1,\ell_{m+1}} = \sum_{i+j = m+1}S^i \Sbar^j (v_{ij}+v'_{ij}) 
+ S^j \Sbar^i (\overline{v}_{ij} + \overline{v'}_{ij}), 
$$
where $v'_{ij}|_D \in \Ker(\square_m +n(m+1)-1 -2(n+1)j)$ and   $v_{ij}|_D$ 
is perpendicular to that kernel.
If there is some $i, j$  ($i+j = m+1$) such that $v'_{ij}|_D \neq 0$,
we use Lemma~\ref{chato} with $k = \ell_{m+1} +1$ and 
$\theta_{ij} = \frac{v'_{ij} \cdot k(m+1)}{-1-ij}$. Note that the constant 
$ \frac{ k(m+1)}{-1-ij}$
was chosen so as to eliminate the kernel term from 
the expression of $u_{m+1, \ell_{m+1}}$. 

Now Lemma~\ref{chato} implies
\begin{multline}
\label{fmprime}
f_{m}' := f_\phi = \left(\sum_{i+j = m+1}S^i \Sbar^j v_{ij} 
+ S^j \Sbar^i \overline{v}_{ij}\right)(-\log||S||_m^2)^{\ell_{m+1}} + \\ +  
\sum_{\ell =0}^{\ell_{m+1}-1} u_{m+1,\ell}(-\log||S||^2_m)^\ell
+\sum_{k'\geq m+1} \sum_{\ell =0}^{\ell_{k'}} u_{k'\ell}(-\log||S||^2_m)^\ell.  
\end{multline}

After Step 1, we can assume (by replacing $f_m$ by $f_m'$ in (\ref{fmprime})) that $f_m$
has an expansion of the form
$$ 
u_{m+1,\ell_{m+1}} = \sum_{i+j = m+1}S^i \Sbar^j (v_{ij}) 
+ S^j \Sbar^i (\overline{v}_{ij}). 
$$

\item{Step 2:}
Now we can solve
$$
\left(\square_m +n(m+1)-1 -2(n+1)j\right)\theta_{ij} = v_{ij}|_D \,\, \, \text{on  } \\  D,
$$
for  $\theta_{ij} \in \Gamma(V_m, L_D^{-i} \bigotimes  \overline{L}_D^{-j})$. Next
let us extend $\theta_{ij} $ to $\Mbar$, and then apply again Lemma \ref{chato}
with $k = \ell_{m+1}$ and $\theta_{ij}$ as above.
The new $f_m$ will have an expansion of the form
$$
f_m = \sum_{\ell =0}^{\ell_{m+1}-1} u_{m+1,\ell}(-\log||S||^2_m)^\ell + O(||S||^{m+2}.
$$
\end{description}

\medskip

By repeating Steps 1 and 2 above, we are able to eliminate all the terms
$\sum_{\ell =0}^{\ell_{m+1}} u_{m+1,\ell}(-\log||S||^2_m)^\ell$ from the 
expansion of $f_m$.
Finally,
let $\phi_m$ be the sum of all functions used in steps 1 and 2,
and define the new metric $||.||_{m+1}$ by letting
$||.||_{m+1}  = e^{-\phi/2}||.||_{m}$. Clearly
the resulting metric satisfies Conditions
1 and 2 of (\ref{expansionfm}). This completes the proof of the
proposition.\qed

\section{Proof of Lemma~\ref{chato}}
\label{s.proofchato}

This section is entirely devoted to proving Lemma~(\ref{chato}) therefore
completing the inductive construction of the metrics $||.||_{m}$.

According to Definition \ref{fphi}, we have  
$
f_\phi(x)=  
f_m - \log\left( \frac{\om_\phi^n}{\om_m^n}\right).
$
Hence, we just need to compute the quotient $\frac{\om_\phi^n}{\om_m^n}$.

Denote by $D_m$ (resp. $D_{\phi}$) the covariant derivative of the metric
$||.||_{m}$ (resp. $||.||_{\phi}$). Similarly, let
$\tilde{\om}_m$ and  $\tilde{\om}_\phi$ denote the corresponding curvature
forms. The following relations are well-known:
\begin{eqnarray}
\label{relations}
D_\phi S & = & D_m S -S\partial \phi \nonumber \\
\tilde{\om}_\phi & = &   \tilde{\om}_m + \frac{\sqrt{-1}}{2\pi}\partial \bar{\partial}\phi
\end{eqnarray}

For simplicity, set $\alpha_m = (-n \log ||S||_m^2)$ and
$\alpha_\phi = (-n \log ||S||_\phi^2) = \alpha_m + n\phi$. 
Then, 
\begin{equation}
\label{e.wmn}
\om_m^n = \alpha_m \tilde{\om}^{n-1}_m \wedge \left( \tilde{\om}_m
+ \frac{n\sqrt{-1}}{2\pi\alpha_m}
\frac{D_m S \wedge \overline{D_m S}}{|S|^2} \right)   \, .
\end{equation}

\begin{multline}
\label{e.wphin}
\om_\phi^n = \alpha_\phi \tilde{\om}^{n-1}_\phi \wedge \left( \tilde{\om}_\phi + 
\frac{n\sqrt{-1}}{2\pi\alpha_\phi}
\frac{D_\phi S \wedge \overline{D_\phi S}}{|S|^2} \right)  = 
(\alpha_m + n\phi)
\left( 
\tilde{\om}_m + \frac{\sqrt{-1}}{2\pi}\partial \bar{\partial}\phi
\right)^{n-1} \wedge 
\\ 
\wedge
\left\{
\left( 
\tilde{\om}_m + \frac{\sqrt{-1}}{2\pi}\partial \bar{\partial}\phi
\right)
 + \frac{n\sqrt{-1}}{2\pi\alpha_\phi} 
\left[
\frac{D_m S \wedge \overline{D_m S}}{|S|^2}  - 
\partial\phi\wedge \frac{\overline{D_m S}}{\Sbar} -
\bar{\partial}\phi \wedge \frac{D_m S}{S} +
\partial\phi\wedge\bar\partial\phi
\right]
\right\} \, .
\end{multline}

\medskip

Some calculations using the definition of $\phi$ lead to
\begin{multline}
\label{delphi}
\partial \phi = 
\sum_{i+j = m+1}
(-\log||S||^2_m)^k
\left(
D_mS^i \Sbar^j \theta_{ij} + D_mS^j \Sbar^i \overline{\theta}_{ij} + S^i
\Sbar^j D_m\theta_{ij} +
 + S^j \Sbar^i D_m\overline{\theta}_{ij})
\right)+
\\
 + k(-\log||S||^2_m)^{k-1} (S^i \Sbar^j \theta_{ij} + S^j \Sbar^i \overline{\theta}_{ij})
\left(
-\frac{D_m S}{S}
\right).
\end{multline}

We will also need the expression for $\partial \bar \partial \phi$.
After some computations using (\ref{delphi}), 
it follows that
\begin{multline}
\label{deldelphi}
\partial \bar{\partial} \phi
= 
(-\log||S||^2_m)^k
\sum_{i+j = m+1}
\left\{
ij(S^i \Sbar^j\theta_{ij} + S^j \Sbar^i\overline{{\theta}}_{ij})
               \left(
\frac{D_m S \wedge \overline{D_m S}}{|S|^2}
               \right) +
\right. \\ \left.
+
(jS^i \Sbar^j D_m \theta_{ij} + iS^j \Sbar^i D_m \overline{\theta_{ij}})
\wedge\frac{\overline{D_m S}}{\Sbar}
+
\frac{{D_m S}}{S} \wedge (iS^i \Sbar^j \overline{D_m}\theta_{ij}
+  jS^j \Sbar^i \overline{D_m \theta_{ij}} ) +
\right. \\ \left.
+
(S^i \Sbar^j D_m \overline{D_m} \theta_{ij} + S^j \Sbar^i 
D_m \overline{D_m}\overline{\theta}_{ij})
+   k(-\log||S||^2_m)^{-1}
\left[ 
(S^i \Sbar^j \overline{D_m}\theta_{ij} + S^j \Sbar^i \overline{D_m}\overline{\theta}_{ij})
\wedge \frac{D_m S}{S}
+ 
\right. \right. \\ \left. \left. 
(S^i \Sbar^j {D_m}\theta_{ij} + S^j \Sbar^i {D_m}\overline{\theta}_{ij})
\wedge \frac{\overline{D_m} \Sbar}{\Sbar}
\right]
 \right\} +
\left(
\frac{k(m+1)\phi}{(-\log||S||^2_m)}
+
\frac{k(k-1)\phi}{(-\log||S||^2_m)^2}
\right)\frac{D_m S \wedge \overline{D_m S}}{|S|^2} \, .
\end{multline}

We can therefore conclude from a simple analysis of (\ref{deldelphi}) that
\begin{equation}
\label{O(2m+2)}
||S||_m^2 (\partial \bar{\partial} \phi)^\ell \wedge {\tilde{\om}_m}^{n-\ell}
 = 
{\tilde{\om}_m}^{n} O(||S||_m^{2m+2}) \, \, \, \text{for $\ell \geq 2$}.
\end{equation}

The above ingredients are going to be needed in the proof of
Lemma~\ref{chato}.

\noindent {\bf Proof of Lemma~\ref{chato}:} Recall that we only need
to compute the quotient $\frac{\om_\phi^n}{\om_m^n}$. Formulas~(\ref{e.wmn})
and~(\ref{e.wphin}) then provide  
\begin{multline}
\frac{\om_\phi^n}{\om_m^n}
 = 
\frac{||S||^2_m \alpha_\phi}{\alpha_m(||S||^2_m + 1/\alpha_m||D_mS||^2_m )\tilde{\om}_m^n} 
\left\{
     \left(\tilde{\om}_m + \frac{\sqrt{-1}}{2\pi}\partial \bar{\partial}\phi\right)^{n-1} \wedge 
\left\{
      \left( \tilde{\om}_m + \frac{\sqrt{-1}}{2\pi}\partial \bar{\partial}\phi \right)
+ \right. \right. \\ \left. \left.
+  \frac{n\sqrt{-1}}{2\pi\alpha_\phi} 
\left[
\frac{D_m S \wedge \overline{D_m S}}{|S|^2}  - 
\partial\phi\wedge \frac{\overline{D_m S}}{\Sbar} -
\bar{\partial}\phi \wedge \frac{D_m S}{S} +
\partial\phi\wedge\bar\partial\phi
\right]
\right\}
\right\}
\\
 = 
\frac{||S||^2_m \alpha_\phi}{(\alpha_m||S||^2_m +||D_mS||^2_m )\tilde{\om}_m^n} \cdot
\left\{
\left[\tilde{\om}_m^{n-1} + 
(n-1)\tilde{\om}_m^{n-2}  \wedge \partial \bar{\partial}\phi \right]
\wedge
\right. \\  \left.
\left[\tilde{\om}_m + \frac{\sqrt{-1}}{2\pi}\partial \bar{\partial}\phi  
+
\frac{n\sqrt{-1}}{2\pi\alpha_\phi} 
\left(
     \frac{D_m S \wedge \overline{D_m S}}{|S|^2}  - 
     \partial\phi\wedge \frac{\overline{D_m S}}{\Sbar} 
      -\bar{\partial}\phi \wedge \frac{D_m S}{S} +
      \partial\phi\wedge\bar\partial\phi                   \right)
\right]
\right\} + O(||S||_m^{m+2})
\end{multline}

Direct calculations using (\ref{delphi}) and its analogous formula for $\bar\partial\phi$ 
lead to (further details can be found in \cite{S})
\begin{equation}
\label{e.deledelbar}
\partial\phi\wedge \frac{\overline{D_m S}}{\Sbar}+ \bar{\partial}\phi \wedge \frac{D_m S}{S} 
 = 
\left( (m+1)\phi +  
\frac{2k\phi}{\log||S||^2_m}\right)\frac{D_m S \wedge \overline{D_m S}}{|S|^2}
+ O(||S||_m^{m+1}).
\end{equation}

Thus it follows that
\begin{multline}
\label{e.blah}
\frac{\om_\phi^n}{\om_m^n}
 = 
\frac{||S||^2_m \alpha_\phi}{\alpha_m(||S||^2_m + 1/\alpha_m||D_mS||^2_m )\tilde{\om}_m^n} \cdot
\left\{
\left[\tilde{\om}_m^{n-1} + 
(n-1)\tilde{\om}_m^{n-2}  \partial \bar{\partial}\phi \right)]
\wedge
\left[\tilde{\om}_m + \frac{\sqrt{-1}}{2\pi}\partial \bar{\partial}\phi  
+
\right. \right. \\ \left. \left.
\frac{n\sqrt{-1}}{2\pi\alpha_\phi} 
\left(1 - (m+1)\phi - \frac{2k\phi}{\log||S||^2_m}
                     \right)
\frac{D_m S \wedge \overline{D_m S}}{|S|^2} + \partial\phi\wedge\bar\partial\phi
\right]
\right\} + O(||S||_m^{m+2}). 
\end{multline}
On the other hand, (\ref{e.deledelbar}) implies that the summand with 
$\partial\phi\wedge\bar\partial\phi$ in (\ref{e.blah})
can be bounded by a multiple of $||S||_m^{m+3} \tilde\om_m^n$. So, we obtain
\begin{multline}
\label{e.blah2}
\frac{\om_\phi^n}{\om_m^n}
 = 
\frac{||S||^2_m \alpha_\phi}{\alpha_m(||S||^2_m + 1/\alpha_m||D_mS||^2_m )\tilde{\om}_m^n} \cdot
\\
\left\{
\tilde{\om}_m^{n-1} \wedge
\left[
\tilde{\om}_m + \frac{\sqrt{-1}}{2\pi}\partial \bar{\partial}\phi +
\frac{n\sqrt{-1}}{2\pi\alpha_\phi} 
\left(1 -(m+1)\phi  -\frac{2k\phi}{\log||S||^2_m}\right)
\frac{D_m S \wedge \overline{D_m S}}{|S|^2}\right] +
\right.  \\ \left. 
+ (n-1)\tilde{\om}_m^{n-2} \wedge
\left( \frac{\sqrt{-1}}{2\pi}\partial \bar{\partial}\phi \right)
\wedge
\left(\tilde{\om}_m + \frac{n\sqrt{-1}}{2\pi\alpha_\phi} \frac{D_m S \wedge
\overline{D_m S}}{|S|^2}\right)
\right\} + O(||S||_m^{m+2}). 
\end{multline}
Also, the definitions of $\alpha_\phi$ and $\alpha_m$ give
\begin{equation}
\label{cool}
\alpha_\phi||S||_m^2  \tilde{\om}_m^{n-1}\wedge 
\left(
\tilde{\om}_m +  \frac{n\sqrt{-1}}{2\pi\alpha_\phi} \frac{D_m S \wedge \overline{D_m S}}{|S|^2}
\right) =  
||S||_m^2 \tilde{\om}_m^{n}\left( \alpha_m +
\frac{||D_mS||_m^2 }{||S||_m^2}\right) + O(||S||_m^{m+2})\tilde{\om}_m^{n}.
\end{equation}
Therefore,
\begin{multline}
\label{e.huge}
\frac{\om_\phi^n}{\om_m^n}
 = 
1+ \frac{||S||^2_m \alpha_\phi}{\alpha_m(||S||^2_m + 1/\alpha_m||D_mS||^2_m )
\tilde{\om}_m^n} \cdot
\\
\left\{
n\tilde{\om}_m^{n-1} \wedge
\left(\frac{\sqrt{-1}}{2\pi}\partial \bar{\partial}\phi\right) -
\phi\left((m+1) + \frac{2k}{(-\log||S||^2_m)} \right)\cdot
\left[
\tilde{\om}_m^{n-1} \wedge
\frac{n\sqrt{-1}}{2\pi\alpha_\phi} 
\frac{D_m S \wedge \overline{D_m S}}{|S|^2} \right]+
\right. \\ \left.
+ (n-1)\tilde{\om}_m^{n-2} \wedge
\left( \frac{\sqrt{-1}}{2\pi}\partial \bar{\partial}\phi \right)
\wedge
\left( \frac{n\sqrt{-1}}{2\pi\alpha_\phi} \frac{D_m S \wedge \overline{D_m S}}{|S|^2}
\right)
\right\} + O(||S||_m^{m+2})=
\\
=
1 - \left( (m+1) + \frac{2k}{(-\log||S||^2_m)}\right)\phi 
+ 
\frac{||S||^2_m \alpha_\phi}{\alpha_m(||S||^2_m + 1/\alpha_m||D_mS||^2_m )\tilde{\om}_m^n} \cdot
\\
\left\{
n\tilde{\om}_m^{n-1} \wedge
\left(\frac{\sqrt{-1}}{2\pi}\partial \bar{\partial}\phi\right) +
(n-1)\tilde{\om}_m^{n-2} \wedge
\left( \frac{\sqrt{-1}}{2\pi}\partial \bar{\partial}\phi \right)
\wedge
\left( \frac{n\sqrt{-1}}{2\pi\alpha_\phi} \frac{D_m S \wedge \overline{D_m S}}{|S|^2}\right)
\right\}
+ O(||S||_m^{m+2})
 = \\ = 
1 - \left((m+1) + \frac{2k}{(-\log||S||^2_m)}\right)\phi 
+ 
\frac{||S||^2_m \alpha_\phi}{\alpha_m(||S||^2_m + 1/\alpha_m||D_mS||^2_m )\tilde{\om}_m^n} 
\cdot
\\ \cdot 
\left\{
(-\log||S||^2_m)^k 
\sum_{i+j=m+1}
\left[
ij(S^i \Sbar^j \theta_{ij} + S^j \Sbar^i \overline{\theta}_{ij}) 
\right]
\left(
\tilde{\om}_m^{n-1} 
\wedge
\left(
\frac{n\sqrt{-1}}{2\pi\alpha_\phi} \frac{D_m S \wedge \overline{D_m S}}{|S|^2}
\right) \right)+ 
\right. \\ \left.
+ (n-1)\tilde{\om}_m^{n-2} \wedge
\left( \frac{\sqrt{-1}}{2\pi}\partial \bar{\partial}\phi \right)
\wedge
\left( \frac{n\sqrt{-1}}{2\pi\alpha_\phi} \frac{D_m S \wedge \overline{D_m S}}{|S|^2}\right)
+
\right. \\ \left. +
\phi\left(\frac{n(m+1)k}{(-\log||S||^2_m)}
+ \frac{n(k-1)k}{(-\log||S||^2_m)^2}\right)
\left(
\tilde{\om}_m^{n-1} \wedge
\frac{n\sqrt{-1}}{2\pi\alpha_\phi} \frac{D_m S \wedge \overline{D_m S}}{|S|^2}
\right)
\right\} + O(||S||_m^{m+2}).
\end{multline}
The last term of (\ref{e.huge})  can be simplified by using 
(\ref{cool}), yielding
\begin{multline}
\frac{\om_\phi^n}{\om_m^n}
 = 
1 + \left(-(m+1)  
+ \frac{k(m-1)}{(-\log||S||^2_m)}
+ \frac{(k-1)k}{(-\log||S||^2_m)^2}\right)\phi
+
\frac{\alpha_\phi||S||^2_m}{\alpha_m(||S||^2_m + 1/\alpha_m||D_mS||^2_m )\tilde{\om}_m^n} \cdot
\\ \cdot
\left\{
 (-\log||S||^2_m)^k
\sum_{i+j=m+1}
\left[ij(S^i \Sbar^j \theta_{ij} + S^j \Sbar^i \overline{\theta}_{ij})\right] 
\tilde{\om}_m^{n-1} \wedge
\left(
\frac{n\sqrt{-1}}{2\pi\alpha_\phi} \frac{D_m S \wedge \overline{D_m S}}{|S|^2}
\right)+ 
\right. \\ \left.
+
 (n-1)\tilde{\om}_m^{n-2} \wedge
\left( \frac{\sqrt{-1}}{2\pi}\partial \bar{\partial}\phi \right)
\wedge
\left( \frac{n\sqrt{-1}}{2\pi\alpha_\phi} \frac{D_m S \wedge \overline{D_m S}}{|S|^2}\right)
\right\}.
\end{multline}

Now observe that the relations
\begin{eqnarray}
\overline{D_m} D_m S^j & = & -D_m \overline{D_m} S^j +jS^j\tilde{\om}_m \nonumber\\
\overline{D_m} D_m\theta_{ij} & = & -D_m \overline{D_m}
\theta_{ij} -(i-j)\theta_{ij}\tilde{\om}_m 
\end{eqnarray}
imply that
\begin{multline}
\sum_{i+j = m+1}
\left(S^i \Sbar^j D_m \overline{D_m} \theta_{ij} + S^j \Sbar^i 
D_m \overline{D_m}\overline{\theta}_{ij}\right)
=  \\
\sum_{i+j = m+1}\left\{\left(S^i \Sbar^j  \overline{D_m} D_m\theta_{ij} +
 S^j \Sbar^i  \overline{D_m}D_m\overline{\theta}_{ij}\right) -\left(S^i \Sbar^j\theta_{ij} -
 S^j \Sbar^i\overline{\theta}_{ij}\right)(i-j) \tilde{\om}_m\right\}. 
\end{multline}

\noindent Hence,
\begin{multline}
\frac{\om_\phi^n}{\om_m^n}
 = 
1 + \left(-(m+1) + 
 \frac{k(m-1)}{(-\log||S||^2_m)}
+ \frac{k(k-1)}{(-\log||S||^2_m)^2}\right)\phi
+
\frac{\alpha_\phi||S||^2_m }{\alpha_m(||S||^2_m + 1/\alpha_m||D_mS||^2_m )\tilde{\om}_m^n} \cdot
 \\ 
\left\{
(-\log||S||^2_m)^k \sum_{i+j=m+1}
[
\alpha_\phi ij(S^i \Sbar^j \theta_{ij} + S^j \Sbar^i \overline{\theta}_{ij})
]
\tilde{\om}_m^{n-1} \wedge
\left(
\frac{n\sqrt{-1}}{2\pi\alpha_\phi} \frac{D_m S \wedge \overline{D_m S}}{|S|^2}
\right)
+\right.\\ \left.
+
(n-1)\tilde{\om}_m^{n-2} \wedge
\left(
\frac{n\sqrt{-1}}{2\pi\alpha_\phi} \frac{D_m S \wedge \overline{D_m S}}{|S|^2}
\right)
\wedge
\left[
-\left(S^i \Sbar^j  \overline{D_m} D_m\theta_{ij} +
 S^j \Sbar^i  \overline{D_m}D_m\overline{\theta}_{ij}\right) 
\right. \right.\\ \left. \left.
+\left(S^i \Sbar^j\theta_{ij} +
 S^j \Sbar^i\overline{\theta}_{ij}\right)(i-j) \tilde{\om}_m. 
\right]
\right\} + 
 O(||S||_m^{m+2}).
\end{multline}
Notice that we can replace the term involving $\alpha_\phi$ from the last formula by 
the analogous term involving $\alpha_m$, since the function $\phi$ is assumed to be of
order $O(||S||_m^{m+1})$. This implies that the residual term of this substitution will lie 
in the term  $ O(||S||_m^{m+2})$.
Therefore, we conclude
\begin{multline}
\frac{\om_\phi^n}{\om_m^n}
=
1 + \left(-(m+1)  
+ \frac{k(m-1)}{(\log||S||^2_m)}
+ \frac{k(k-1)}{(\log||S||^2_m)^2}\right)\phi
-
\\
- (-\log||S||^2_m)^k 
\sum_{i+j = m+1}\left(S^i \Sbar^j  \square_m\theta_{ij} +
 S^j \Sbar^i  \overline{\square}_m\overline{\theta}_{ij}\right) 
+ (-\log||S||^2_m)^k \cdot \\ \cdot \sum_{i+j = m+1}\alpha_m ij
+ (n-1)((m+1)-2j)
\left(S^i \Sbar^j\theta_{ij} +
 S^j \Sbar^i\overline{\theta}_{ij}
\right) + O(||S||_m^{m+2})
= \\ =
1 + n(m+1)\phi +  \frac{k}{(-\log||S||^2_m)}
\left(  \frac{k-1}{(-\log||S||^2_m)} + (m-1) \right)
+
\\
+
\sum_{i+j = m+1}\left(S^i \Sbar^j  \square_m\theta_{ij} +
 S^j \Sbar^i  \overline{\square}_m\overline{\theta}_{ij}\right)
+
2(n-1)
 (-\log||S||^2_m)^k \sum_{i+j = m+1} j \left(S^i \Sbar^j\theta_{ij} +
 S^j \Sbar^i\overline{\theta}_{ij}\right) +
\\
+  (\log||S||^2_m)^{k-1} \sum_{i+j = m+1} ij \left(S^i \Sbar^j\theta_{ij} -
 S^j \Sbar^i\overline{\theta}_{ij}\right) +
O(||S||_m^{m+2}).
\end{multline}
Finally,
\begin{multline}
f_\phi
= f_m  - \log\left(\frac{\om_\phi^n}{\om_m^n}\right)=
nm\phi +  \frac{k \phi}{(-\log||S||^2_m)}
\left(  \frac{k-1}{(-\log||S||^2_m)} - (m-1) \right)+ 
\\ +
 (-\log||S||^2_m)^{k}
\sum_{i+j = m+1}
\left[
\left(S^i \Sbar^j  \square_m\theta_{ij} +
 S^j \Sbar^i  \overline{\square}_m\overline{\theta}_{ij}\right)
+2(n+1) j \left(S^i \Sbar^j\theta_{ij} -
 S^j \Sbar^i\overline{\theta}_{ij}\right)
\right]
+ \\
+(-\log||S||^2_m)^{k-1} \sum_{i+j = m+1} ij \left(S^i \Sbar^j\theta_{ij} -
 S^j \Sbar^i\overline{\theta}_{ij}\right)  + O(||S||_m^{m+2}),
\end{multline}
which proves the lemma. The inductive construction
of the metrics $||.||_m$ is also completed.\qed

\section{Complete \kahler Metrics on $M$}
\label{s.whole}

In this section we shall finish the proof of Theorem~\ref{t.approximatingmetrics}.
In particular it is going to be necessary to consider the asymptotic
behavior of the Riemann curvature tensor.

For each $m \geq 1$, consider the function $f_m$ constructed in Section~\ref{s.approx}.
For this choice, let the corresponding
$$
\om_m = 
\frac{\sqrt{-1}}{2\pi(n+1)} \partial \bar{\partial}
(-n\log||S||_m^2)^{ \frac{n+1}{n}}
$$
define a $(1,1)-$form on $M$.
If $\delta_m$ is sufficiently small, 
$\om_m$ is positive definite on 
\hbox{$V_m = \{ ||S(x)|| \leq \delta_m \}$}, and defines a \kahler metric 
$g_m$. 

\begin{lemma}
\label{l.complete}
The \kahler manifolds $(V_m, \partial V_m, g_m)$
are all complete, equivalent to each other near $D$, and for each
$m>0$, the function 
$$
\rho= \frac{2}{n+1}(-n\log||S||_m^2)^{\frac{n+1}{2n}} 
$$
is equivalent to any distance function from a fixed point 
in $V_m$ near $D$. 
\end{lemma}
\begin{proof}
Fix $m>0$. We have
$$
|\nabla_m \rho|^2_{g_m} = \frac{\sqrt{-1}}{2\pi} 
\frac{\partial \rho \wedge \bar\partial \rho \wedge \om_m^{n-1}}{\om_m^{n}}.
$$
Since 
$$
\partial \rho = n^{\frac{n+1}{2n}} (-\log ||S||^2_m)^{\frac{1-n}{2n}} \frac{D_mS}{S}
$$
it follows that
\begin{equation}
\frac{\sqrt{-1}}{2\pi} \partial \rho \wedge \bar\partial \rho \wedge \om_m^{n-1}
 = 
 (-n\log ||S||^2_m)^{\frac{1-n}{2n}}
(-n\log ||S||^2_m)^{\frac{n-1}{2n}}  \tilde{\om}_m^{n-1} \wedge
\frac{\sqrt{-1}}{2\pi}
\frac{D_m S \wedge \overline{D_m S}}{|S|^2}.
\end{equation}
Therefore, we conclude that
\begin{equation}
|\nabla_m \rho|^2_{g_m} =
\frac{1}{n}
\frac{ \tilde{\om}_m^{n-1}\wedge
\frac{n\sqrt{-1}}{2\pi}
\frac{D_m S \wedge \overline{D_m S}}{|S|^2}}{(-n\log ||S||^2_m)\tilde{\om}_m^{n} +  
\tilde{\om}_m^{n-1}\wedge
\frac{n\sqrt{-1}}{2\pi}
\frac{D_m S \wedge \overline{D_m S}}{|S|^2}}
 = 
\frac{1}{n}
\frac{||D_mS||_m^2}{(-n\log ||S||^2_m) + ||D_mS||_m^2}.
\end{equation}
Now,
recall that $||D_mS||_m^2$
is never zero, and that
$\lim_{||S||_m \rightarrow 0} (-\log||S||^2_m \cdot ||S||^2_m) = 0$. Hence
$$ 
\xymatrix{ |\nabla_m \rho|^2_{g_m} \ar[rr] ^{||S||_m \rightarrow 0} & &
 \frac{1}{n}}, 
$$
proving that $\rho$ is equivalent to any distance function from the boundary near $D$.

Also, since $\rho \rightarrow \infty$ when $x \rightarrow D$, the K\"ahler manifold 
 $(V_m, \partial V_m, g_m)$ is complete.

We claim that all the metrics $g_m$ are equivalent near $D$. To check
the claim, note first that each
$\tilde\om_m$ is the curvature form of the metric $||.||_m$, 
hence, for every $m, \ell \in \NN$,  $\tilde\om_m$ is equivalent to $\tilde\om_\ell$ 
near $D$. The claim then follows from (\ref{relate}),
that relates the expressions for $\om_m$ and
$\tilde\om_m$. 

Finally here is a remark about the volume growth of $(V_m, \partial V_m , g_m)$:
since $\om_m^n$ is equivalent to $\tilde\om_m^n (-n\log ||S||_m^2)$, it suffices
to consider the integral
$$
\int_{||S(x)||_m \geq e^{-1/2n \rho^{\frac{2n}{n+1}}}} (-n\log ||S||_m^2)\tilde\om_m^n,
$$
which is of order $\rho^{\frac{2n}{n+1}}$.
\end{proof}

In the sequel we are going to carry out the estimates of the Riemann
curvature tensor $R(g_m)$ corresponding to the metric $g_m$ which are
involved in the statement of Theorem~(\ref{t.approximatingmetrics}). Let us
begin with the following lemma:

\begin{lemma}
\label{l.curvature}
Let  $(V_m, \partial V_m, g_m)$ be complete,  \kahler manifolds with 
boundary defined as in Lemma \ref{l.complete}.
Then the norm of $R(g_m)$ with respect to the metric $g_m$ decays at the order of
at least $(-n\log||S||_m^2)^{\frac{-1}{n}}$ near $D$.
\end{lemma}

\begin{proof}
We shall prove the statement in local coordinates, as follows.
There exists a finite covering $U_t$ of $D$ in $\Mbar$ such that for each $t$, 
there is a local uniformization $\Pi_t: \tilde{U_t} \rightarrow U_t$ such that 
${\Pi_t}^{-1}(D)$ is smooth in $ \tilde{U_t}$. The covering $U_t$ can,
in addition, be chosen so that given a local coordinate system
$(z_1, \dots, z_n)$ in  $ \tilde{U_t}$, with $S = z_n$ and $z' = (z_1, \dots, z_{n-1})$
defining coordinates along $D$, we have
\begin{multline}
\label{curvtensor}
\sum_{i,j,k,l = 1}^n R({\Pi_t}^{*}(g_m))_{i\bar{j}k\bar{l}}(z', z_n)\xi^i \xi^{\bar{j}}\xi^k 
\xi^{\bar{l}}
 = \\ =
 (-n\log|z_n|^2 )^{1/n} 
\sum_{i,j,k,l = 1}^n R({\Pi_t}^{*}(g_D|_{{\Pi_t}^{-1}(D)}))_{i\bar{j}k\bar{l}}
\xi^i \xi^{\bar{j}}\xi^k \xi^{\bar{l}} 
+ O((-n\log|z_n|^2 )^{-1/n} ),
\end{multline}
for every $g_m$-unit vector $(\xi^1, \dots, \xi^n)$, where $g_D$ is the  \kahler
metric defined by the restriction of the curvature form $\tilde{\om}$ to the divisor.

Without loss of generality, we assume that $U_t\cap\Mbar$ is smooth.

For every $x \in U_t\cap M$, consider local coordinates $(z_1, \dots, z_n)$
for a neighborhood of $x$ which satisfy the conditions below:
\begin{itemize}
\item 
The defining section $S$ of the divisor is given by $z_n$.
\item 
The curvature form $\tilde{\om}_m$ of $||.||_m$ is represented in the
mentioned coordinates by
the tensor $(h_{i\bar{j}})$ where $(h_{i\bar{j}})$
satisfies
$$
h_{i\bar{j}}(x) = \delta_{ij}; \hspace{0.5cm}
\frac{\partial h_{i\bar{j}}}{\partial z_k}(x)= 0 \hspace{.3cm}\text {if  $j<n$};
 \hspace{0.5cm}
\frac{\partial h_{i\bar{j}}}{\partial \bar{z_l}}(x)= 0 \hspace{.3cm} \text {if $i<n$};
$$
\item 
The hermitian metric  $||.||_m$ is represented by a positive funcion 
$a$ with $a(x)=1$, $da(x) = 0$ and $d(\frac{\partial a}{\partial z_k})(x) = 0$.
\end{itemize}
 
In order to simplify notation, let us write $B= B(|z_n|) = (-n\log|z_n|^2)$, and let
us drop the subscripts for the metric $g_m$ to be denoted by $g$
from now on.

Formula~(\ref{relate}) implies that 
$$
{g}^{i\jbar} (x ) =
\begin{cases}   
O(B^{-1/n}) & \text {if  $i=j$ and $ i<n$} \\ 
o(B^{-1/n})  & \text{if $i\neq j$ and  $i,j <n$} \\ 
O(|z_n|^2B^{-1/n}) &  \text{if $ i = j = n$},
\end{cases}
$$
and  computations give (check \cite{S} for further details)
$$
\frac{\partial g_{i\jbar}}{\partial z_k} =  
B^{1/n}
   \left[
          \frac{\partial h_{i\jbar}}{\partial z_k} - \frac{1}{z_nB}
           \left(
                  \de_{kn} h_{i\jbar} + \de_{in} h_{k\jbar} +  
                  \de_{in} \de_{jn} \de_{kn}
      \left( \frac{n-1}{B} + \frac{1}{|z_n|^2}   \right)
           \right)
   \right]
$$
and
\begin{multline}
\frac{\partial^2 g_{i\jbar}}{\partial z_k \partial \zbar_l} = 
B^{1/n}
\left\{
       \frac{\partial^2 h_{i\jbar}}{\partial z_k \partial \zbar_l}
-\frac{1}{\zbar_n B} 
      \left(\de_{ln} \frac{\partial h_{i\jbar}}{\partial z_k} + 
            \de_{jn} \frac{\partial h_{i\lbar}}{\partial z_k}    \right) 
-\frac{1}{z_n B} 
      \left(\de_{kn} \frac{\partial h_{i\jbar}}{\partial \zbar_l} + 
            \de_{in} \frac{\partial h_{k\jbar}}{\partial \zbar_l}    \right)
+  \right. \\ +  \left.
 \frac{1-n}{|z_n|^2 B^2}
      \left( \de_{ln}(\de_{kn}  h_{i\jbar} + \de_{in}  h_{k\jbar} )
+  
\de_{jn}(\de_{kn}  h_{i\lbar} + \de_{in}  h_{k\lbar} )     \right)
\right. +\\ \left.
 + \frac{ \de_{in} \de_{jn} \de_{kn} \de_{ln}}{|z_n|^4 B^3}
     \left( |z_n|^2(n-1)(1-2n)+ (1-n)B +B^2 \right)
\right\} \, .
\end{multline}
If $(\xi^1, \dots,\xi^n )$ is a $g$-unit tangent vector, then 
$$
\begin{cases}
\label{estxi}
|\xi^i|^2 \leq C B^{-1/n}& \text {if  $i < n$, and} \\
|\xi^n|^2 \leq C |z_n|^2 B^{(n-1)/n}, & 
\end{cases}
$$
where $C$ is a constant that does not depend neither on 
the unit vector  $(\xi^1, \dots,\xi^n )$ nor on the point
$x \in D$.
Hence, in local coordinates, we have
\begin{multline}
\label{e.curvtensor}
 R(g)_{i\jbar k \lbar}(x)(\xi^i \bar{\xi}^j \xi^k \bar{\xi}^l) (x)
= 
\left[
\frac{\partial^2 g_{i\jbar}}{\partial z_k \partial \zbar_l} (x)
+ 
\sum_{u,v = 1}^n g^{u\vbar}(x)\frac{\partial g_{i\vbar}}{\partial z_k}(x)
\frac{\partial g_{u\jbar}}{\partial \zbar_l}(x)
\right](\xi^i \bar{\xi}^j \xi^k \bar{\xi}^l) = \\
 = 
-B^{1/n}
\left\{
\frac{\partial^2 h_{i\jbar}}{\partial z_k \partial \zbar_l} (x)
(\xi^i \bar{\xi}^j \xi^k \bar{\xi}^l)
 + 
\frac{4(1-n)|\xi^n|^2}{|z_n|^2B^2} \left(\sum_{i=1}^n |\xi^i|^2 \right)
\right. \\ \left.
+
\frac{|\xi^n|^4}{|z_n|^4B^3} \left( |z_n|^2(n-1)(1-2n)+ (1-n)B +B^2 \right)
\right\}+
\\
+ 
B^{2/n} \sum_{u,v = 1}^n g^{u\vbar}(x)
\left[
- \frac{\xi^n}{z_nB}
   \left(  2\xi^i h_{i\vbar} + \xi^n \de_{i\vbar}\left(\frac{n-1}{B} + \frac{1}{|z_n|^2} \right) 
           \right)
\right] \cdot 
\\ \cdot
\left[
- \frac{\bar{\xi}^n}{z_nB}
   \left(  2\bar{\xi}^j h_{u\jbar} + \bar\xi^n \de_{u\nbar} \left(\frac{n-1}{B} + 
\frac{1}{|z_n|^2} \right) 
           \right)
\right]
\end{multline}

Next let us separately bound each of the above terms.
By using (\ref{estxi}) and 
our previous choice of local coordinates,
we obtain, when $z_n$ approaches zero,
$$
\frac{4(1-n)|\xi^n|^2}{|z_n|^2B^2} \left(\sum_{i=1}^n |\xi^i|^2 \right)
\leq  \frac{C |z_n|^2B^{(n-1)/n}}{|z_n|^2B^2} (B^{-1/n} + |z_n|^2B^{(n-1)/n})\leq
C B^{-(n+2)/n},
$$
where $C$ denotes a uniform constant.
Also, we have that
$$
\frac{|\xi^n|^4}{|z_n|^4B^3} \left( |z_n|^2(n-1)(1-2n)+ (1-n)B +B^2 \right)
\leq C B^{-1/n}(B^{-1} + 1) \leq  C B^{-1/n}.
$$
Now, notice that the expression
\begin{equation}
\frac{\xi^n}{z_nB}
   \left(  2\xi^i h_{i\vbar} + \xi^n \de_{i\vbar}\left(\frac{n-1}{B} + \frac{1}{|z_n|^2} \right) 
           \right)
\leq
C(B^{-(n+2)/2n} + |z_n|^{-1}B^{-1/n}),
\end{equation}
needs special attention due to the presence of a term  involving $|z_n|^{-1}$.
However our  estimate  for 
$$
g^{u\vbar} = B^{-1/n}\left[(1-\de_{un}\de_{vn})O(1) + \de_{un}\de_{vn}O(|z_n|^2) \right]
$$
shows that this term is compensated  by the last term of the above expression.
The estimate for the decay of the last term in (\ref{e.curvtensor}) is analogous to the case
discussed above, and will be ommited. For details, please refer to \cite{S}.

In conclusion, we have 
$$
 R(g)_{i\jbar k \lbar}(x)(\xi^i \bar{\xi}^j \xi^k \bar{\xi}^l) (x)
= 
B^{1/n}
\frac{\partial^2 h_{i\jbar}}{\partial z_k \partial \zbar_l} (x)
(\xi^i \bar{\xi}^j \xi^k \bar{\xi}^l)
+ O(B^{-1/n}),
$$
which implies (\ref{curvtensor}), and concludes the proof of the lemma.
\end{proof}

The reader may also notice that
Lemma~\ref{l.curvature} completes the proof of Proposition~\ref{t.inductive}.

\medskip

\begin{lemma}
\label{l.derivcurv}
For the manifold $(V_m, \partial V_m, g_m)$ defined in 
Lemma~\ref{l.complete}, we have 
\begin{equation}
\label{e.derivcurv}
||\nabla^k R(g_m)||_{g_m} (x)= O \left( \rho_m^{-\frac{k+2}{n+1}}(x) \right),
\end{equation}
where $\rho_m(x)$ is the distance function from a fixed point associated to $g_m$.
\end{lemma}

\begin{proof}
The proof of this lemma will follow the same idea as the proof of 
\cite{TY2}, Lemma 2.5.

We start by fixing $m$, and fixing a small $\de >0$ such that 
$V_\de := \{ x \in M ; ||S||_m < \de \} \subset V_m$.
In order to prove the result,  we are going to introduce a suitable
new coordinate system on $V_\de$.

Because $D$ is admissible, it follows that the 
total space of the unit sphere bundle of $L_D|_D$
(with respect to the metric $||.||_m$)
is a smooth manifold of real dimension $2n+1$, to be denoted by $M_1$.

Since $L_D$ is simply the normal bundle of $D$ in $\Mbar$, 
there exists a diffeomorphism 
$$
\Psi:  M_1 \times (0, \de) \rightarrow V_\de 
$$ 
induced by the exponential map of $(\Mbar, ||.||_m)$ along $D$.

It is also known that the \kahler form of $g_m$ is given by
$$
\om_m = \frac{\sqrt{-1}}{2\pi} \frac{n^{1 + 1/n}}{n+1}  \partial \bar{\partial}
(-\log(||S||^2 e^{2\phi_m}))^{ \frac{n+1}{n}},
$$
where $\phi_m$ is a smooth function on $\Mbar$, that can be written as
$\sum_{\kappa \geq m} \sum_{\ell =0}^{\ell_\kappa} u_{\kappa\ell}(-\log||S||^2_k)^\ell$,
where $u_{\kappa\ell}$ are smooth functions on $\overline{V}_m$ that vanish to order 
$\kappa$ on $D$.

Combining the facts above, the pull-back of $g_m$ under $\Psi$ on
$M_1 \times (0, \de)$ is given by
\begin{eqnarray}
\label{e.pullbackg3}
\nonumber
\Psi^*g_m &=& 
(-n \log(||S||^2))^{\frac{1}{n}} 
g(||S||, ||S|| \log(||S||)) 
+ \\ 
\nonumber
&+& 
(-n \log(||S||^2))^{\frac{1-n}{n}} d\left( (-n \log(||S||^2))^{\frac{1}{n}}
\right)h(||S||, ||S|| \log(||S||)) 
+  \\ &+& 
(-n \log(||S||^2))^{\frac{1-2n}{n}} d\left( (-n \log(||S||^2))^{\frac{1}{n}} \right)^2 
u(||S||, ||S|| \log(||S||)) \, .
\end{eqnarray}
Here $g(.,.)$ (resp. $h(.,.)$, $u(.,.)$) is a $C^\infty$ family of metrics
(resp. $1-$tensors, functions)
on $M_1$, such that for each fixed integer $\ell>0$, there exists a constant 
$K_\ell$ that bounds all covariant derivatives (with respect to a fixed metric 
$\tilde{h}$ on $M_1$) of $g(t_0,t_1)$ (resp. $h(t_0,t_1)$, $u(t_0,t_1)$) 
up to order $\ell$, for every $t_0 \in [0, \de]$ and $t_1 \in [0, \de \log(\de)]$.  

Setting $\rho = (-n \log(||S||^2))^{\frac{n+1}{2n}},$ (\ref{e.pullbackg3}) 
becomes 
\begin{equation}
\label{e.pullbackg3gamma}
\Psi^*g_m = \rho^{\frac{2}{n+1}}g(.,.)+ \rho^{\frac{2(1-n)}{n+1}}d(\rho^{\frac{2}{n+1}})h(.,.)
+  \rho^{\frac{2(1-2n)}{n+1}}d(\rho^{\frac{2}{n+1}})^2u(.,.),
\end{equation}
and hence we can regard $\Psi^*g_m$ as being a metric defined on 
$M_1 \times \left((-n \log\de^2 )^{\frac{n+1}{2n}}, \infty \right)$.

It is easy to see that, if $g$ is any Riemannian metric, and 
$\tilde g = \lambda^2 g$ is a rescaling of $g$, then 
the norm of the covariant derivatives of the new metric, taken with 
respect to the new metric $\tilde g$, satisfy 
$||\tilde{\nabla}^{k} R(\tilde g)||_{\tilde g} = \lambda^{-(k+2)} ||\nabla^k R(g)||_g$.
Therefore, if we define a new metric $\tilde g$ on 
$M_1 \times \left((-n \log\de^2 )^{\frac{n+1}{2n}}, \infty \right)$ by 
$\tilde g = \rho_m^{-\frac{2}{n+1}}(x) \Psi^*g_m $,  showing  
(\ref{e.derivcurv}) is equivalent to showing that 
$||\tilde{\nabla}^{k} R(\tilde g)||_{\tilde g} (x)  = O(1)$.

But this last statement follows clearly from the assertions on 
$g(.,.), h(.,.)$ and $u(.,.)$.
\end{proof}



In the same lines of Lemma~\ref{l.derivcurv}, one can show

\begin{lemma}
\label{l.derivfm}
For a fixed $m$, let  $f_m$ be the smooth function constructed in Section~\ref{s.approx}.
Then,
\begin{equation}
\label{e.derivfm}
||\nabla^k f_m||_{g_m} = O(||S||_m^{m+1} \rho_m^{-\frac{k}{n+1}}),
\end{equation}
where we recall that $\rho_m(.)$ is the distance function associated to 
the metric $g_m$.
\end{lemma}

\begin{proof}
We will use the same notation and objects defined on the proof
of Lemma~\ref{l.derivcurv}.

Begin by observing that (\ref{e.derivfm}) is equivalent to showing that
\begin{equation}
\label{e.derivcurvmod}
||\tilde{\nabla}^k f_m||_{\rho_m^{-\frac{2}{n+1}}\Psi^*g_m} (\Psi^{-1}(x))
= O_k(1)(||S||_m^{m+1})(x),
\end{equation}
where $0_k(1)$ is a quantity bounded by a constant that depends on $k$, and
$\tilde\nabla$ is the covariant derivative associated to the metric 
$\rho_m^{-\frac{2}{n+1}}\Psi^*g_m$ on $M_1 \times (0, \de)$.

Also, recall that the function $f_m$ satisfies 
$\Ric(g_m) - \Om  = \frac{\sqrt{-1}}{2\pi} \partial \bar \partial f_m$ on 
$V_\de$ , for a fixed $\Om \in c_1(-K_{\Mbar} \otimes -L_D)$.

Hence, (\ref{e.derivcurvmod}) follows clearly from the estimates on 
the covariant derivatives of the curvature tensor $R(\rho_m^{\frac{-2}{n+1}}\Psi^*g_m)$
given by Lemma~\ref{l.derivcurv}, and the observation on the local expression
of the metric $\rho_m^{\frac{-2}{n+1}}\Psi^*g_m$. 

\end{proof}

\bigskip

We are finally able to prove Theorem~\ref{t.approximatingmetrics}.

\noindent {\bf Proof of Theorem~\ref{t.approximatingmetrics}:} In what follows
we keep the preceding setting and notations.

Since the divisor $D$ is assumed to be ample in $\Mbar$, there exists a 
hermitian metric $||.||'$ on $L_D$ whose curvature form ${\tilde\om}'$ is 
positive definite on $D$.

Fix an integer $k \geq \ve$, and
write, for $\ve >0$,
\begin{equation}
\label{e.omegag}
\om_{g_\ve} = \om_k + C_\ve \frac{\sqrt{-1}}{2\pi}\partial \bar\partial(-||S||')^{\ve},
\hspace{2cm} C_\ve> 0,
\end{equation}
 where $\om_k$ is the 
\kahler form defined on Section~2.
The \kahler form $\om_{g_\ve}$ is positive definite on $M$, and gives rise to a  
complete \kahler metric $g_\ve$ on $M$.

Let $\de >0$ be such that $V_\de = \{ ||S(x)|| <  \delta\} \subset V_k$. 
On $V_\de$, $\om_{g_\ve}$ satisfies
$\text{Ric}(g_\ve) - \Om = \frac{\sqrt{-1}}{2\pi}\partial \bar{\partial}f$, and we want 
to estimate the decay of $f$ at infinity.

On the other hand, on  $V_\de$, we have 
$\text{Ric}(g_k) - \Om = \frac{\sqrt{-1}}{2\pi}\partial \bar{\partial}f_k$ which,
in turn, implies that
\begin{eqnarray}
\label{e.relationfandfm}
\nonumber
f  = f_k - \log \frac{\om_{g_\ve}^n}{\om_k^n} &=& f_k - 
\log \left(\frac{w^n_k + C_\ve ||S||'^{(\ve -1)} 
w^{n-1}_k \wedge \frac{n\sqrt{-1}}{2\pi}(D'S\wedge\bar{D}'\bar{S})}{\om_k^n}\right)
=\\ &=& f_k-
C_\ve ||S||'^{(\ve -1)} ||D'S||_{g_k},
\end{eqnarray}
where $D'S$ denotes the covariant derivative of the metric $||.||'$.
Hence, in order to estimate the decay of $f$, it suffices to study the decay of 
$ ||D'S||_{g_k}$.

In what follows, we will use  the notation defined in the proof of 
Lemma~\ref{l.derivcurv}: recall that $\Psi$ is the map between a tubular 
neighborhood of $D$ inside $\Mbar$ induced by the exponential map of $||.||$ 
along $D$.

Let the function $\ga := \Psi^*(||S||')^{\ve}$ be defined on 
$M_1 \times\left((-n \log\de^2 )^{\frac{n+1}{2n}}, \infty \right) $.
Our goal is to understand the decay of $ ||D'S||_{g_k}$, which is equivalent to
studying the decay of 
$|\tilde{\nabla} \ga|_{ \rho^{-2/(n+1)}\Psi^*g_k}(\Psi^{-1}(x))$, where 
$\tilde{\nabla} $ denotes the covariant derivative of the metric 
$\rho^{-2/(n+1)}\Psi^*g_k$.

Notice that on $M_1 \times\left((-n \log\de^2 )^{\frac{n+1}{2n}}, \infty \right) $,
the function $\ga$ can be written under the form
$$
e^{\tilde\ga}(., \exp\{\rho^{\frac{2n}{n+1}}\})\exp\{\frac{\ve}{n}\rho^{\frac{2n}{n+1}}\} ,
$$
where $\tilde\ga$ is a smooth function on
$M_1 \times\left((-n \log\de^2 )^{\frac{n+1}{2n}}, \infty \right)$
with all derivatives bounded in terms of a fixed product metric.

Hence, from (\ref{e.pullbackg3gamma}), it follows that 
\begin{equation}
\label{e.weirdO}
|\tilde{\nabla} \ga|_{ \rho^{-2/(n+1)}\Psi^*g_k}(\Psi^{-1}(x))
= O (\exp\{\frac{\ve}{n}\rho^{\frac{2n}{n+1}}\}),
\end{equation}
since the curvature tensor of $\rho^{-2/(n+1)}\Psi^*g_k$
is bounded near $\Psi^{-1}(x)$.

Note also that (\ref{e.weirdO}) is equivalent to 
$$
||D'S||_{g_k} = O(||S||^{\ve}), 
$$
which shows that the metric $g_\ve$, with corresponding \kahler
form $\om_{g_\ve}$ defined by (\ref{e.omegag}), satisfies the equation
$\Ric(g_\ve)- \Om = \partial\bar\partial f_\ve$, for $f_\ve$ a smooth function
that decays at least to the order of $O(||S||^{\ve})$. 

To close the proof of Theorem~(\ref{t.approximatingmetrics}), it only remains 
to note that the curvature estimates for the new metric $g_\ve$
follow easily from the estimates on the curvature tensor $R(g_m)$, 
described in Lemma~\ref{l.curvature}.\qed

\section{Asymptotics of the Monge-Amp\`ere Equation on $M$}
\label{s.MA}

This last section is intended to provide the proof of
Theorem~\ref{t.mainthm}.

Let $(M, g)$ be a complete \kahler manifold, with \kahler form 
$\om$.
Consider the following Monge-Amp\`ere equation on $M$:
\begin{equation}
\label{MA}
\begin{cases}
\left(
\om + \frac{\sqrt{-1}}{2\pi}\partial\bar\partial u
\right)^n  = e^f \om^n, & \\
\om + \frac{\sqrt{-1}}{2\pi}\partial\bar\partial u >0,
& \text{$u \in C^\infty(M, \RR)$},
\end{cases}
\end{equation}
where $f$ is a given smooth function
satisfying the integrability 
condition
\begin{equation}
\label{intcondition}
\int_M (e^f-1)w^n = 0.
\end{equation}

Whenever $u$ is a solution to (\ref{MA}), the $(1,1)$-form
$\om + \frac{\sqrt{-1}}{2\pi}\partial\bar\partial u$ satisfies
$Ric(\om + \frac{\sqrt{-1}}{2\pi}\partial\bar\partial u) = f$.
So, in order to define metrics with prescribed Ricci curvature,
it is enough to determine a solution to (\ref{MA}).

In \cite{TY1}, Tian and Yau proved that (\ref{MA}) has, in fact,
solutions modulo assuming certain conditions on the volume growth of $g$ as well
as on the decay of $f$ at infinity. For the convenience of the reader, we
state here their main result.

\begin{thm} {\bf(Tian, Yau, \cite{TY1})}
\label{t.TY1}
Let $(M, g)$ be a complete \kahler manifold, satisfying:
\begin{itemize}
\item
The sectional curvature of the metric $g$ is bounded by a constant $K$.
\item 
$\text{Vol}_g(B_R(x_0))\leq CR^2$ for all $R>0$ and 
$\text{Vol}_g(B_1(x_0)) \geq C^{-1} (1+ \rho(x))^{-\beta}$, for 
a constant $\beta$, where $\text{Vol}_g$ denotes the volume associated 
to the metric $g$, $B_R(x_0)$ is the geodesic ball of radius $R$ about
a fixed point $x_0 \in M$, and $\rho(x)$ denotes the distance (with respect to $g$)
from $x_0$ to $x$.
\item
There are positive numbers $r>0$, $r_1>r_2>0$ such that for every
$x\in M$, there
exists a holomorphic map $\phi_x : \cU_x \subset (\CC^n, 0) \rightarrow B_r(x)$ such that 
$\phi_x(0)= x$. Here one has $B_{r_2} \subset \cU_x\subset  B_{r_1}$, 
where $B_r:= \{ z \in \CC^n; |z|\leq r \}$. Furthermore $\phi_x^*g$ is a \kahler metric
in $\cU_x$ whose metric tensor has derivatives up to order $2$ bounded and 
$1/2$-H\"older-continuously bounded.
\end{itemize}

Let $f$ be a smooth function, satisfying the integrability condition (\ref{intcondition})
and such that 
\begin{equation}
\sup_M \{ |\nabla_g f|, |\Delta_g f| \} \leq C \hspace{2cm} |f(x)| \leq C(1+\rho(x))^{-N},
\end{equation}
for some constant $C$, for all $x$ in $M$, where $N\geq 4+2\beta$.

Then there exists a bounded, smooth solution $u$ for (\ref{MA}), such that 
$\om  + \frac{\sqrt{-1}}{2\pi}\partial\bar\partial u$ defines a complete 
\kahler metric equivalent to $g$.
\end{thm}

An interesting question 
addressed to by Tian and Yau in the same paper is that whether or not
the resulting metric
is asymptotically as close to $g$ as possible modulo
assuming further conditions on the decay of $f$. 
We provide an answer to this problem in the remainder of this paper.

We are interested in studying the Monge-Amp\`ere equation (\ref{MA}) for the 
\kahler manifold $(M, \om_{g_\ve})$ constructed in Section~4. 
More precisely, given $\ve >0$, we want to understand the asymptotic behavior
of a solution $u$ to the problem
\begin{equation}
\label{MAm}
\begin{cases}
\left(
\om_{g_\ve} + \frac{\sqrt{-1}}{2\pi}\partial\bar\partial u
\right)^n  = e^{f_{g_\ve}} \om_{g_\ve}^n, & \\
\om_{g_\ve} + \frac{\sqrt{-1}}{2\pi}\partial\bar\partial u >0,
& \text{$u \in C^\infty(M, \RR)$}.
\end{cases}
\end{equation}

To guarantee the existence of a solution to (\ref{MAm}), 
we need to check that the function $f_{g_\ve}$ (defined on
Theorem \ref{t.approximatingmetrics}) satisfies the integrability condition 
(\ref{intcondition}).

\begin{lemma} 
There exists a number $\lambda > 0$ such that, by replacing 
$\phi$ by $\phi + \la$ in   
Definition (\ref{fphi}) of $f_\phi$, we have
\begin{equation}
\label{intconditionlemma}
\int_M (e^{f_{g_\ve}}-1)\om_{g_\ve}^n = 0.
\end{equation}
\end{lemma}

\begin{proof}
Recall the definition of $\om_{g_\ve}$:
\begin{equation}
\om_{g_\ve} = \om_{\phi}
 + C_\ve \frac{\sqrt{-1}}{2\pi}\partial \bar\partial(-||S||')^{2\ve},
\end{equation}
where $\om_{\phi} = \frac{\sqrt{-1}}{2\pi} \frac{n^{1 + 1/n}}{n+1} \partial \bar{\partial}
(-\log||S||_\phi^2)^{ \frac{n+1}{n}}$.
We choose
$\phi$  (as in Section~2) so that the corresponding 
$f_\phi$ decays faster than $O(||S||^{\ve})$.

A direct computation using integration by parts 
shows that  $ \int_M \om_{g_\ve}^n - \om_\phi^n = 0$.
Also, the definitions of $\om_{g_\ve}$ and $\om_\phi$ imply that
$e^{f_{g_\ve}}\om_{g_\ve}^n = e^{f_\phi}\om_\phi^n$. Therefore, 
$
\int_M (e^{f_{g_\ve}}-1)\om_{g_\ve}^n 
=
\int_M (e^{f_{\phi}}-1)\om_{\phi}^n. 
$
On the other hand, Definition \ref{fphi} gives 
\begin{equation}
\label{e.invariant}
e^{f_{\phi}} \om_\phi^n  = 
\frac{e^{-\Psi} {\om'}^n}{||S||^2}.
\end{equation}
Notice that the function $\Psi$ remains unchanged if we
replace $\phi$ by $\phi + \la$, since $\tilde \om_\phi =
\tilde \om_{\phi + \la}$. 
Therefore, the right-hand side of (\ref{e.invariant}) is
invariant under the transformation $\phi \mapsto \phi + \la$.

On the other hand, a straightforward computation using (\ref{omegaphi}) shows that
\begin{equation}
\om_{\phi + \la}^n  =  
\left(
\frac{\sqrt{-1}}{2\pi} \frac{n^{1 + 1/n}}{n+1} \partial \bar{\partial}
(-\log||S||_{\phi + \la}^2)^{\frac{n+1}{n}}
\right) = 
\left(
\frac{\sqrt{-1}}{2\pi} \frac{n^{1 + 1/n}}{n+1} \partial \bar{\partial}
(-\log||S||_{\phi}^2)^{\frac{n+1}{n}}
\right) - n \la \tilde{\om}_\phi^n,
\end{equation}
where $ \tilde{\om}_\phi^n$ is the curvature form of the 
hermitian metric $||.||_\phi$.

Therefore, by redefining $f_\phi$ by 
$
f_\phi = - \log ||S||^2 - \log \frac{\om_{\phi + \la}^n}{\om'^n} - \Psi, 
$
we have that
\begin{equation}
\label{e.vanish}
\int_M (e^{f_{\phi}} -1) \om_{\phi+ \la}^n
=
\int_M  \left(\frac{e^{-\Psi} {\om'}^n}{||S||^2} -  \om_{\phi}^n \right)
- n \la \int_M  \tilde{\om}_\phi^n.
\end{equation}
Since the first integral in the above expression is finite and independent
on $\la$, we can choose the number $\la$ so as to make the right-hand side of
(\ref{e.vanish}) equal to {\it zero}. This establishes the lemma.
\end{proof}

The previous lemma shows that each $f_{g_\ve}$ satisfies the conditions in 
the existence theorem of Tian and Yau. 
Also, the estimates on the decay of the Riemann  curvature tensor (Lemma~\ref{l.curvature})
and the observation on the volume growth of the metric $g_\ve$ (see the remark after
Lemma~\ref{l.complete}) show that $(M, g_\ve)$ is a complete \kahler manifold
in which Theorem~\ref{t.TY1} can be applied. 

Therefore, for each $\ve > 0$, there exists a bounded and smooth solution 
$u_\ve$ to the problem (\ref{MAm}). Our goal now is to understand the 
asymptotic behavior of $u_\ve$.

Denote by $\om_\Om$ the \kahler form on $M$ given by Theorem \ref{t.TY1}, when we use
$g_\ve$ (given by Theorem \ref{t.approximatingmetrics}) as the ambient metric:
$$
\om_\Om = \om_{g_\ve} + \frac{\sqrt{-1}}{2 \pi} \partial \bar \partial u_\ve.
$$
Clearly, it suffices to prove the asymptotic assertions on $u_\ve$ for a small tubular 
neighborhood of $D$ in $\Mbar$. Recall from Theorem \ref{t.approximatingmetrics} that on 
$V_\ve \setminus D$,
$$
\om_{g_\ve} = \om_m + C_\ve  \frac{\sqrt{-1}}{2 \pi} \partial \bar \partial
\left( || S ||' \right)^{2\ve},
$$
for some $m \geq  \ve$ fixed.

Since $\om_m$ and $\om_{g_\ve}$ are cohomologous, there exists a function $u_m$ 
such that we can write, in $V_m \setminus D$,
\begin{equation}
\label{e.omOm}
\om_\Om = \om_m +  \frac{\sqrt{-1}}{2 \pi} \partial \bar \partial u_m.
\end{equation}
On the other hand, if $f_m$ is the function defined by (\ref{fphi}),
(\ref{e.omOm}) implies that $u_m$ satisfies 
\begin{equation}
\label{e.MAomm}
\left(
\om_m +  \frac{\sqrt{-1}}{2 \pi} \partial \bar \partial u_m
\right)^n
= 
e^{f_m}\om_m^n  \hspace{1cm} \text{on $V_m \setminus D$,}
\end{equation}
where we remind the reader that $|f_m|_{g_m}$ is of order of $O(||S||^m_m)$.

\begin{lemma}
\label{l.barrier}
On the neighborhood \hbox{$V_m \setminus D = \{ 0< ||S||_m < \de_m \}$,}
we have 
\begin{multline}
   \left\{
\om_m + \frac{\sqrt{-1}}{2\pi}\partial \bar \partial\left(C 
\left[S^i \Sbar^j \theta_{ij} + \Sbar^i S^j \bar\theta_{ij} \right]
(-n\log(||S||_m^2))^k \right)
  \right\}^n
 = \\ = 
\om_m^n 
\left[
1  + C (-n\log(||S||_m^2))^{k-\frac{n+1}{n}} 
\left\{ 
ij(-n\log(||S||_m^2))^2  \left[S^i \Sbar^j \theta_{ij} + \Sbar^i S^j \bar\theta_{ij} \right]
- \right. \right.\\ \left.\left. -
(-n\log(||S||_m^2)) \left[ 
\left( k(i+j) +j(n-1)\right) S^i \Sbar^j \theta_{ij}
+ 
\left( k(i+j) +i(n-1)\right)  \Sbar^i S^j \bar\theta_{ij}
\right]
\right. \right. \\ \left. \left.
+ k(k-n)
\right\}
+ O(||S||_m^{i+j+1})
\right],   
\end{multline}
where  $\theta_{ij}$ is a $C^\infty$ local section of $L_D^{-i} \otimes  \overline{L}_D^{-j}$
on $V_m$.
\end{lemma}

\begin{proof}
In order to simplify the notation, define $B = (-n\log(||S||_m^2))$.
Computations (that can be found in detail in \cite{S}) lead to
\begin{multline}
\label{e.long}
\frac{\sqrt{-1}}{2\pi}\partial \bar \partial
\left( 
C B^k \left[S^i \Sbar^j \theta_{ij} + \Sbar^i S^j \bar\theta_{ij} \right] 
\right) = 
\tilde{\om}_m 
    \left\{ CB^{k-1}\left[ \left( -jB + k\right)S^i \Sbar^j \theta_{ij} \right]
                   + \left[ \left( -iB + k\right) \Sbar^i S^j \bar\theta_{ij}\right]
    \right\} 
+ \\+ 
  \frac{\sqrt{-1}}{2\pi} \frac{D_mS \wedge \overline{D_m S}}{|S|^2} 
\left\{ C B^{k-2}  \left[S^i \Sbar^j \theta_{ij} + \Sbar^i S^j \bar\theta_{ij} \right] 
\left[ ijF^2  -k(i+j)F +k(k-1) \right] \right\}
+\\ + 
CB^{k-1}\frac{\sqrt{-1}}{2\pi}
\left[ \left(jB-k \right)S^i \Sbar^j D_m\theta_{ij}
+  \left(iB-k \right)\Sbar^i S^j D_m\bar\theta_{ij}
\right]\wedge \frac{\overline{D_m S}}{\Sbar}
+ \\ + 
CB^{k-1}\frac{\sqrt{-1}}{2\pi}\frac{D_m S}{S} \wedge
\left[ \left(iB-k \right)S^i \Sbar^j {\bar D}_m\theta_{ij}
+  \left(jB-k \right)\Sbar^i S^j {\bar D}_m\bar\theta_{ij}
\right]
+ \\ + 
C B^k  \frac{\sqrt{-1}}{2\pi} \left[S^i \Sbar^j D_m {\bar D}_m\theta_{ij} 
+ \Sbar^i S^j D_m {\bar D}_m\bar\theta_{ij} \right], 
\end{multline}
where $D_m$ stands for the covariant derivative with respect to the hermitian metric $||.||_m$, 
and where $\tilde{\om}_m $ is its corresponding curvature form.


From (\ref{relate}), we have
$$
\om_m = (-n\log(||S||_m^2))^{\frac{1}{n}} \tilde\om_m + (-n\log(||S||_m^2))^{\frac{1-n}{n}} 
\frac{\sqrt{-1}}{2\pi} \frac{D_mS \wedge \overline{D_m S}}{|S|^2} \, .
$$
Thus
\begin{multline}
\label{e.eachpowerbarrier}
 \left[
\om_m + \frac{\sqrt{-1}}{2\pi}\partial \bar \partial\left(C B^k 
 \left[S^i \Sbar^j \theta_{ij} + \Sbar^i S^j \bar\theta_{ij} \right]\right)
  \right]^n = 
\left[
a \tilde\om_m + b \frac{\sqrt{-1}}{2\pi} \frac{D_mS \wedge \overline{D_mS}}{|S|^2}
 + \right. \\  \left.
 \frac{\sqrt{-1}}{2\pi}\left(c_1 D_m \theta_{ij} + c_2 D_m\bar\theta_{ij}\right)
\wedge \frac{\overline{D_m S}}{\Sbar}
+ \frac{\sqrt{-1}}{2\pi}
\frac{{D_m S}}{S}  \wedge \left(d_1{\bar D}_m \theta_{ij} + d_2 {\bar D}_m \bar\theta_{ij}\right)
 +
e \frac{\sqrt{-1}}{2\pi} {\bar D}_mD_m \theta_{ij}
\right]^n,
\end{multline}
where 
\begin{eqnarray}
\nonumber
a &=& B^{\frac{1}{n}} \left[ 1 - C B^{k-\frac{n+1}{n}}\left[ (jB-k) S^i \Sbar^j \theta_{ij} 
 (iB -k)\Sbar^i S^j \bar\theta_{ij} + \right]  \right] 
\\
\nonumber
b &=&   B^{\frac{1-n}{n}}\left[ 1 + C B^{k-\frac{n+1}{n}} 
 \left[S^i \Sbar^j \theta_{ij} + \Sbar^i S^j \bar\theta_{ij} \right]
\left[ ijF^2  -k(i+j)F +k(k-1) \right] \right]
\\
c_1 &=& C  S^i \Sbar^j B^{k-1}[jB-k], \hspace{1cm} 
c_2 =  C  \Sbar^i S^j B^{k-1}[iB-k] 
\\
\nonumber
d_1 &=& C  S^i \Sbar^j B^{k-1}[iB-k], \hspace{1cm} 
d_2 =  C  \Sbar^i S^j B^{k-1}[jB-k] 
\\
\nonumber
e &=& C  S^i \Sbar^j B^{k}.
\end{eqnarray}

Now, we proceed on estimating each term on (\ref{e.eachpowerbarrier}).

\begin{equation}
\label{e.an}
a^n \tilde\om_m^n = 
 B^{\frac{1}{n}} \left[ 1 - C n B^{k-\frac{n+1}{n}}\left[ (jB-k) S^i \Sbar^j \theta_{ij} 
+ (iB -k)\Sbar^i S^j \bar\theta_{ij}  \right]   + O(||S||_m^{i+j+1})
\right] 
\tilde\om_m^n \, .
\end{equation}
Also
\begin{multline}
\label{e.secondtermone}
n a^{n-1}b  
 = 
\left[ 1 - C (n-1) B^{k-\frac{n+1}{n}}\left[ (jB-k) S^i \Sbar^j \theta_{ij} +
 (iB -k)\Sbar^i S^j \bar\theta_{ij} + \right]  \right] 
\cdot \\ \cdot
\left[ 1 + C B^{k-\frac{n+1}{n}} 
 \left[S^i \Sbar^j \theta_{ij} + \Sbar^i S^j \bar\theta_{ij} \right]
\left[ ijB^2  -k(i+j)B +k(k-1) \right] \right]
= \\ = 
1  + C B^{k-\frac{n+1}{n}} 
\left\{ 
ijB^2  \left[S^i \Sbar^j \theta_{ij} + \Sbar^i S^j \bar\theta_{ij} \right]
- \right. \\ \left. -
B \left[ 
\left( k(i+j) +j(n-1)\right) S^i \Sbar^j \theta_{ij}
+ 
\left( k(i+j) +i(n-1)\right)  \Sbar^i S^j \bar\theta_{ij}
\right]
+ k(k-n)
\right\}
+ O(||S||_m^{i+j+1}) \, .
\end{multline}
The expressions for the other terms are analogous, and will henceforth 
be omitted.

From (\ref{omegaphi}), we deduce that
$$
\tilde\om_m^{n} = \om_m^n \frac{||S||_m^2 B^{-1}}{||S||_m^2 + B^{-1}||D_mS||_m^2}, 
$$
and since 
$$
 \frac{||S||_m^2 }{||S||_m^2 + B^{-1}||D_mS||_m^2} = 
 \frac{||S||_m^2 B^{-1}}{||D_mS||_m^2} 
\left( \frac{1}{1 + \frac{B||S||_m^2}{||D_mS||_m^2}}\right)= O(||S||_m^2B^{-1}),
$$
all the terms in (\ref{e.eachpowerbarrier}) will decay at the order of
at least $ O(||S||_m^{i+j+1})$, with the exception of the term 
(\ref{e.secondtermone}), which will be written as: 
\begin{multline}
a^{n-1} b \left\{\tilde\om_m^{n-1} \wedge \frac{n\sqrt{-1}}{2\pi} 
\frac{D_mS \wedge \overline{D_m S}}{|S|^2} \right\}
= a^{n-1} b  \frac{||D_mS||^2_m}{||S||_m^2} {\tilde\om_m^{n}} = 
\\
=
a^{n-1} b  \left(\frac{||D_mS||^2_m}{||S||_m^2} \right)\frac{ \om_m^n||S||_m^2 B^{-1}}{||S||_m^2 
+ B^{-1}||D_mS||_m^2}
=  
\om_m^n 
\left[
1  + C B^{k-\frac{n+1}{n}} 
\left\{ 
ijB^2  \left[S^i \Sbar^j \theta_{ij} + \Sbar^i S^j \bar\theta_{ij} \right]
- \right. \right.\\ \left.\left. -
B \left[ 
\left( k(i+j) +j(n-1)\right) S^i \Sbar^j \theta_{ij}
+ 
\left( k(i+j) +i(n-1)\right)  \Sbar^i S^j \bar\theta_{ij}
\right]
+ k(k-n)
\right\}
+ O(||S||_m^{i+j+1})
\right]   \, .
\end{multline}
Therefore, 
\begin{multline}
   \left[
\om_m + \frac{\sqrt{-1}}{2\pi}\partial \bar \partial\left(C S^i \Sbar^j \theta_{ij}B^k \right)
  \right]^n
=  
\om_m^n 
\left[
1  + C B^{k-\frac{n+1}{n}} 
\left\{ 
ijB^2  \left[S^i \Sbar^j \theta_{ij} + \Sbar^i S^j \bar\theta_{ij} \right]
- \right. \right.\\ \left.\left. -
B \left[ 
\left( k(i+j) +j(n-1)\right) S^i \Sbar^j \theta_{ij}
+ 
\left( k(i+j) +i(n-1)\right)  \Sbar^i S^j \bar\theta_{ij}
\right]
+ k(k-n)
\right\}
+ O(||S||_m^{i+j+1})
\right],   
\end{multline}
completing the proof of the lemma.
\end{proof}

\begin{prop}
\label{p.uvanishesalot}
Let $u_m$ be a solution to the Monge-Amp\`ere equation (\ref{e.MAomm}).
If $u_m(x)$ converges uniformly to zero as $x$ approaches the divisor, 
then there exists a constant $C = C(m)$ such that
\begin{equation}
\label{e.uvanishesalot}
|u_m(x)| \leq C ||S||_m^{m+1} \hspace{1cm} \text{on $V_m\setminus D$}.
\end{equation}
\end{prop}

\begin{proof}
It suffices to prove (\ref{e.uvanishesalot}) in a neighborhood of $D$.
Apply Lemma \ref{l.barrier} for $i= m+2$ and $j= -1$, and choose the section
$\theta_{ij}$ so that the function $S^i \Sbar^j \theta_{ij} + \Sbar^i S^j \bar\theta_{ij}$
is positive on $V_m\setminus D$. Note that there is, in fact, a
$C^{\infty}$-section $\theta_{ij}$ satisfying this condition. Indeed, a local
section on a trivializing coordinate can clearly be constructed by means of a
bump function. In particular we can consider finitely many local sections as above
such that the union of their supports covers the all of $D$.
Since the positivity condition is naturally respected by the
cocycle relations arising from the change of coordinates, the desired section
$\theta_{ij}$ can simply be obtained by adding these local sections.

With the above choices, we have
\begin{multline}
   \left[
\om_m + \frac{\sqrt{-1}}{2\pi}\partial \bar \partial
C \left[S^i \Sbar^j \theta_{ij} + \Sbar^i S^j \bar\theta_{ij} \right]
(-n\log(||S||_m^2))^k
  \right]^n
=  \\ = 
\om_m^n 
\left[
1  - C (m+2) (-n\log(||S||_m^2))^{k+\frac{n-1}{n}} 
\left\{\left[ 1 + o(1) \right] 
  \left[S^i \Sbar^j \theta_{ij} + \Sbar^i S^j \bar\theta_{ij} \right]
\right\}
+ O(||S||_m^{m+2})
\right].  
\end{multline}

On the other hand,
\begin{equation}
e^{f_m}\om_m^n = [1 + O(||S||_m^{m+1})]\om_m^n \hspace{1cm} \text{on $V_m\setminus D$}. 
\end{equation}
More precisely, we can write on $V_m\setminus D$
\begin{equation}
e^{f_m}\om_m^n = [1 + 
\sum_{\ell =0}^{\ell_{m+1}} 
\{S^i \Sbar^j \theta_{ij} + \Sbar^i S^j \bar\theta_{ij}\}(-\log||S||^2_m)^\ell
+  O(||S||_m^{m+2})]\om_m^n, 
\end{equation}
for sections $\theta_{ij} \in \Gamma(V_m\setminus D, L_D^{-i} \otimes {\bar L}_D^{-j})$.

Let $\ve>0$, and define $C_i = \frac{C'_i}{\ve}$, 
where $C_1':= \sup_{x \in V_m\setminus D}(|u_m| + 1)$, and
$C_2' = - C_1'$.
Then, for all $x\in V_m\setminus D$ verifying 
\begin{equation}
\label{e.nghd}
\left( S^{m+2} \Sbar^{-1} \theta_{m+2,-1} + \Sbar^{m+2} S^{-1} \bar\theta_{m+2, -1}
(-n\log(||S||_m^2))^{\ell_{m+1}}\right) (x)
 = \ve,
\end{equation}
it follows that
$$C_1 \left( S^{m+2} \Sbar^{-1} \theta_{m+2,-1} + \Sbar^{m+2} S^{-1} \bar\theta_{m+2, -1}
(-n\log(||S||_m^2))^{\ell_{m+1}}\right) (x)> |u_m(x)|.
$$
Furthermore, if $\ve$ is sufficiently small, then on the subset 
$$\{x \in V_m\setminus D; \left( S^{m+2} \Sbar^{-1} \theta_{m+2,-1}
+ \Sbar^{m+2} S^{-1} \bar\theta_{m+2, -1}
(-n\log(||S||_m^2))^{\ell_{m+1}}\right) (x)\leq \ve \},
$$ 
we have
(for $i = m+2$ and $j = -1$)
\begin{equation}
\nonumber
 \left[
\om_m + \frac{\sqrt{-1}}{2\pi}\partial \bar \partial
C_1 \left[S^i \Sbar^j \theta_{ij} + \Sbar^i S^j \bar\theta_{ij} \right]
(-n\log(||S||_m^2))^{\ell_{m+1}-\frac{n-1}{n}}
  \right]^n \leq e^{f_m}\om_m^n, \hspace{0.3cm} \text{and}
\end{equation}
\begin{equation}
\nonumber
 \left[
\om_m + \frac{\sqrt{-1}}{2\pi}\partial \bar \partial
C_2 \left[S^i \Sbar^j \theta_{ij} + \Sbar^i S^j \bar\theta_{ij} \right]
(-n\log(||S||_m^2))^{\ell_{m+1}-\frac{n-1}{n}}
  \right]^n \geq e^{f_m}\om_m^n.
\end{equation} 
Finally, by using the hypothesis on the uniform vanishing of $u_m$ on $D$, the proposition 
follows from the maximum principle for the complex Monge-Amp\`ere
 operator: we obtain the following bound (write $\ell = \ell_{m+1}-\frac{n-1}{n}$)
\begin{equation}
C_2 \left[S^i \Sbar^j \theta_{ij} + \Sbar^i S^j \bar\theta_{ij} \right]
(-n\log(||S||_m^2))^{\ell}
\leq
u_m
\leq
C_1 \left[S^i \Sbar^j \theta_{ij} + \Sbar^i S^j \bar\theta_{ij} \right]
(-n\log(||S||_m^2))^{\ell}
\end{equation}
on the neighborhood given in (\ref{e.nghd}), which completes the proof
of the proposition.
\end{proof}

\medskip

Now, let us deal with another
 step in the proof of Theorem \ref{t.mainthm}, which 
consists of showing that the solution to the Monge-Amp\`ere equation
(\ref{e.MAomm}) actually converges  uniformly to zero.

\begin{prop}
\label{p.uvanishes}
For a fixed $m \geq 2 $, let $u_m$ be a solution to (\ref{e.MAomm}). Then
$u_m(x)$ converges uniformly
to zero as $x$ approaches the divisor $D$.
\end{prop}
\begin{proof}
In \cite{TY1}, the solution $u_m$ to the Monge-Amp\`ere equation~(\ref{MA}) is
obtained as the uniform limit, as $\ve$ goes to {\it zero}, of solutions $u_{m,\ve}$  
of the perturbed Monge-Amp\`ere equations
\begin{equation}
\label{MAve}
\begin{cases}
\left(
\om_m + \frac{\sqrt{-1}}{2\pi}\partial\bar\partial u
\right)^n  = e^{f_m + \ve u} \om_m^n, & \\
\om_m + \frac{\sqrt{-1}}{2\pi}\partial\bar\partial u >0,
& \text{$u \in C^\infty(M, \RR)$}.
\end{cases}
\end{equation}

On the neighborhood $V_m$,
Lemma~(\ref{l.barrier}) applied for $i = 2$, $j= -1$ and $k = 0$, gives
\begin{multline}
   \left\{
\om_m + \frac{\sqrt{-1}}{2\pi}\partial \bar \partial\left(C 
\left[S^2 \Sbar^{-1} \theta_{2, -1} + \Sbar^2 S^{-1} \bar\theta_{2, -1} \right]\right)
  \right\}^n
 = \\ = 
\om_m^n 
\left[
1  -2 C (-n\log(||S||_m^2))^{-\frac{n+1}{n}} 
\left\{ 
(-n\log(||S||_m^2))^2  
\left[S^2 \Sbar^{-1} \theta_{2, -1} + \Sbar^2 S^{-1} \bar\theta_{2, -1} \right]
- \right. \right.\\ \left.\left. -
(-n\log(||S||_m^2)) \left[ 
(1-n) S^2 \Sbar^{-1} \theta_{2, -1}
+ 
2(n-1)  \Sbar^2 S^{-1} \bar\theta_{2, -1}
\right]
\right\}
+ O(||S||_m^{2})
\right] \, .
\end{multline}

Again, we can choose appropriate local $C^\infty$-sections $\theta_{2, -1}$ such that
$$\left[S^2 \Sbar^{-1} \theta_{2, -1} + \Sbar^2 S^{-1} \bar\theta_{2, -1} \right]$$
is a positive function on a neighborhood of the divisor, and use this function as a
uniform barrier to the sequence of solutions $\{u_{m,\ve}\}$.

Note that $e^{f_m + \ve u_{m, \ve}} = 1 +  O(||S||_m)$. Hence, as in the proof
of Lemma \ref{l.barrier}, we can define, for a fixed $\de >0$,  $C_i = \frac{C_i'}{\de}$, 
where $C'_1:= \sup_{x \in V_m\setminus D}(|u_{m,\ve}| + 1)$ and $C_2' = -C_1'$.
A priori, $C_i'$ could depend on $\ve$, but it turns out (see \cite{TY1} for details) 
that $\sup_M |u_{m,\ve}|$ can be bounded 
uniformly by a constant independent on $\ve$.
Then, for all $x\in V_m\setminus D$ such that 
$$
\left( S^{2} \Sbar^{-1} \theta_{2,-1} + \Sbar^{2} S^{-1} \bar\theta_{2, -1}\right) (x)
 = \de,
$$
we have that
$$
C_1 \left( S^{2} \Sbar^{-1} \theta_{2,-1} + \Sbar^{2} S^{-1} \bar\theta_{2, -1}\right) (x)> |u_{m,\ve(x)}|.
$$
In addition, in the neighborhood 
$\{x \in V_m\setminus D; \left( S^{2} \Sbar^{-1} \theta_{2,-1} + \Sbar^{2} S^{-1}
\bar\theta_{2, -1}\right) (x)\leq \de \}$, for a fixed $\de$ sufficiently small, we have
\begin{equation}
\nonumber
 \left[
\om_m + \frac{\sqrt{-1}}{2\pi}\partial \bar \partial
C_1\left[S^2 \Sbar^{-1} \theta_{2, -1} + \Sbar^2 S^{-1} \bar\theta_{2, -1} \right]
 \right]^n \leq 
e^{f_m + \ve C_1\left[S^2 \Sbar^{-1} \theta_{2, -1} + \Sbar^2 S^{-1} \bar\theta_{2, -1} \right]}\om_m^n
\end{equation}
and
\begin{equation}
\nonumber
 \left[
\om_m + \frac{\sqrt{-1}}{2\pi}\partial \bar \partial
C_2\left[S^2 \Sbar^{-1} \theta_{2, -1} + \Sbar^2 S^{-1} \bar\theta_{2, -1} \right]
 \right]^n \geq 
e^{f_m + \ve C_1\left[S^2 \Sbar^{-1} \theta_{2, -1} + \Sbar^2 S^{-1} \bar\theta_{2, -1} \right]}\om_m^n.
\end{equation}

Since $u_{m,\ve}$ vanishes at $D$ (see \cite{CY1}), we can apply the maximum principle
to conclude that there exists a $C$ independent of $\ve$ such that, near $D$,
$$
- C\left[S^2 \Sbar^{-1} \theta_{2, -1} + \Sbar^2 S^{-1} \bar\theta_{2, -1} \right] 
\leq u_{m,\ve}  \leq C\left[S^2 \Sbar^{-1} \theta_{2, -1} + \Sbar^2 S^{-1} \bar\theta_{2, -1} 
\right]
$$
Now, since the neighborhood $\{x \in V_m\setminus D; \left( S^{2} \Sbar^{-1} \theta_{2,-1} + \Sbar^{2} S^{-1}
\bar\theta_{2, -1}\right) (x)\leq \de \}$ is a fixed set, independent of $\ve$, and
the constant $C$ above is also independent of $\ve$, we can 
pass to the limit when $\ve$ goes to zero, obtaining the claim.
\end{proof}

\bigskip

We can still provide further information about the decay of the covariant derivatives of the solution
$u_m$ to (\ref{e.MAomm}).

\begin{prop}
\label{p.derivum}
For a fixed $m> 1$, let $u_m$ be a solution to (\ref{e.MAomm}). Then, there exists 
$C = C(k, m)$ such that, for all $x \in V_m \setminus D$, 
\begin{equation}
\label{e.derivum}
|\nabla^k u_m|_{g_m}(x) \leq C ||S||^m_m (x) \rho_m^{-\frac{k+2}{n+1}}.
\end{equation}
\end{prop}

\begin{proof}
The strategy of this proof will be induction on $k$, and 
the use of Moser Iteration Technique.

Recall that Lemma~\ref{l.derivcurv} gives
$$
||\nabla^k R(g_m)||_{g_m} (x) = O(\rho^{-\frac{k+2}{n+1}})
$$
for all $x \in V_m \setminus D$.

In order to use Moser Iteration, we will need to localize on a
small ball around a point $x \in V_m \setminus D$. 
Let us define a new rescaled \kahler metric on 
$B_R(x, g_m)$ by $\tilde(g) = R^{-2}g_m$, where
$R = \rho_m^{\frac{1}{n+1}}(x)$.
Note that $B_1(x, \tilde(g)) = B_R(x, g_m)$.

Then, 
$$
||\tilde \nabla^k R(\tilde g )||_{\tilde g} (x) =
||\nabla^k R(g_m)||_{g_m} (x) R^{k+2},
$$
and hence $\sup \{ ||\tilde \nabla^k R(\tilde g )||_{\tilde g} (x); x \in _{B_1(x, \tilde(g))}\} \leq C$.

By writing $\tilde u = R^{-2} u_m$, we get that $\tilde u$ satisfies
\begin{equation}
\label{e.MAtilde}
\left(
\om_{\tilde g} + \frac{\sqrt{-1}}{2 \pi} \partial \bar \partial \tilde u
\right)^n
=
e^{f_m} \om_{\tilde g}^n
\hspace{2cm} \text{on $B_1(x, \tilde(g))$,}
\end{equation}
and from Propositions \ref{p.uvanishes} and \ref{p.uvanishesalot}, we 
also have that
$\sup \{|\tilde u|; x \in B_1(x, \tilde(g)) \} \leq C ||S||_m^m (x).$

Lemma~\ref{l.derivfm} also tells us that 
$|\nabla^k f_m|_{g_m}(x) = O(||S||^{m+1}_m (x) \rho_m^{-\frac{k}{n+1}}(x))$, 
hence there exists a constant $C = C(k, m)$ such that 
$\sup \{|\nabla^k f_m|_{g_m}; x \in B_1(x, \tilde(g)) \} \leq C ||S||_m^m (x).$

Now, we have all the ingredients to start the inductive proof. 
Since $|\nabla^k u_m|_{g_m}(x) = R^{-k+2}|\nabla^k \tilde u|_{\tilde g}(x)$,
it suffices to prove that there exists a constant $C$ such that
\begin{equation}
\label{e.conclusion}
|\nabla^k \tilde u|_{\tilde g}(x) \leq C ||S||_m^m (x).
\end{equation}

We will proceed by induction on $k$. First, consider the case $k = 1$.
By the second order estimate in \cite{Y1} (see also \cite{TY1}),  
we have that the \kahler metric which solves the Monge-Amp\'ere equation
is equivalent to the chosen representative of the fixed \kahler class, 
{\em ie}, there exists a constant $C = C_m$ such that
$$
0\leq \om_m + \frac{\sqrt{-1}}{2\pi} \partial \bar \partial u_m \leq C \om_m . 
$$
This implies that 
$$
0\leq \om_{\tilde g} + \frac{\sqrt{-1}}{2\pi} \partial \bar \partial \tilde u
\leq C \om_{\tilde g} \hspace{2cm} \text{on $B_1(x, \tilde(g))$}. 
$$

Let $\eta$ be a cut-off function that vanishes outside the unit ball, 
and is identically one on $B_{1/2}(x, \tilde(g))$. By multiplying 
both sides of (\ref{e.MAtilde}) by $\eta^2 \tilde u$
and integrating by parts, we obtain the $L^2$- estimate of $|\nabla \tilde u|$:
\begin{equation}
\int_{B_{1/2}(x, \tilde(g))} |\nabla \tilde u|^2_{\tilde g} \om_{\tilde g}^n
\leq
C \int_{B_{1}(x, \tilde(g))}|e^{f_m} - 1||\tilde u| \om_{\tilde g}^n \leq ||S||^{2m}_m(x).
\end{equation}

Now, differentiating (\ref{e.MAtilde}) with respect to 
$z_k$, $1\leq k \leq n $, we obtain
\begin{equation}
\label{e.moser1}
{\tilde g}^{i \jbar}\left( \frac{\partial \tilde u}{ \partial z_k}\right)_{i \jbar}
=
(e^{f_m} -1){\tilde g}^{i \jbar} \left( \frac{\partial {\tilde g}_{i\jbar}}{ \partial z_k}\right)_{i \jbar}
+
\frac{\partial f_m}{\partial z_k} e^{f_m} + O(|\nabla^2 \tilde u|),
\end{equation}
where $O(|\nabla^2 \tilde u|)$ denotes the terms that can be bounded by $|\nabla^2 \tilde u|$.

Our estimates on the $L^2$-norm of $|\nabla \tilde u|$ on the ball ${B_{1/2}(x, \tilde g )}$
allow us to apply Theorem $8.17$, \cite{GT} (the Moser Iteration Method) to conclude that
\begin{equation}
\label{e.conclusionk1}
|\nabla \tilde u|_{\tilde g} (x) \leq C ||S||^m_m(x).
\end{equation}

Recall that $x$ was a point chosen arbitrarily, so this estimate holds for all   
$x \in V_m \setminus D$.

Now, the next step is to multiply both sides of (\ref{e.conclusionk1})
by $\eta^2 \frac{\partial \tilde u}{\partial z_k}$ (for $\eta$ an adequate
cut-off function defined on ${B_{1}(x, \tilde g)}$) and integrate by parts, 
in order to obtain the estimate on the $L^2$-estimate on the norm of 
$\nabla^2 \tilde u$:
\begin{equation}
\int_{B_{1/2}(x, \tilde(g))} |\nabla^2 \tilde u|^2_{\tilde g} \om_{\tilde g}^n
 \leq C ||S||^{2m}_m(x).
\end{equation}
Since it is analogous to our previous step, we will omit further explanations.
We note that the pointwise estimate on $|\nabla \tilde u|$ allow us
to obtain the $L^2$-estimate on $|\nabla^2 \tilde u|$. This completes the first step 
on the induction (for $k = 1$.

Now, assume that the following is true for all $j\leq k-1$:
\begin{equation}
\label{e.conclusionkj}
|\nabla^j \tilde u|_{\tilde g} (x) \leq C ||S||^m_m(x), \hspace{2cm} 
\text{and}
\end{equation}
\begin{equation}
\int_{B_{1/2}(x, \tilde(g))} |\nabla^k \tilde u|^2_{\tilde g} \om_{\tilde g}^n
 \leq C ||S||^{2m}_m(x).
\end{equation}

Differentiating (\ref{e.MAtilde})$k$ times, we obtain
\begin{equation}
\label{e.moserkj}
{\tilde g}^{i \jbar}\left( \frac{\partial^k \tilde u}{ \partial z^{i_1} \dots  
\partial z^{i_k}}\right)_{i \jbar}
=
\frac{\partial^k f_m}{\partial z^{i_1} \dots  
\partial z^{i_k}} e^{f_m} + O(||S||_m^m (x)).
\end{equation}

Our estimates on the decay of $|\nabla^k f_m|_{\tilde g}$ (given by 
Lemma~\ref{l.derivfm}) allow us to use again Moser Iteration 
to conclude that
\begin{equation}
\label{e.conclusionk}
|\nabla^k \tilde u|_{\tilde g} (x) \leq C ||S||^m_m(x), 
\end{equation}
completing the proof of the Proposition.
\end{proof}

\noindent {\bf Proof of Theorem~\ref{t.mainthm}:} It follows immediately
from the combination of
Propositions~\ref{p.uvanishesalot},  ~\ref{p.uvanishes} and 
\ref{p.derivum}.
\hfill \fbox{}

\vskip 1cm

\flushleft

\medskip
{\bf Bianca Santoro} \ \  (bsantoro@msri.org)\\
Massachusetts Institute of Technology\\
Department of Mathematics \\
77 Massachusetts Avenue\\
Cambridge - MA - 02139\\
USA\\

\vspace{1cm}

Mathematical Sciences Research Institute\\
17 Gauss Way, room 217\\
Berkeley, CA  94720\\
USA\\

\end{document}